\documentclass[12pt,eqno]{article}
\usepackage{amssymb}
\usepackage[leqno]{amsmath}
\usepackage[pdftex]{graphicx}
\usepackage{epsfig}
\usepackage{epstopdf}
\usepackage{algorithm}
\usepackage{algorithmic}
\usepackage{array}
\usepackage{caption,subcaption}
\usepackage{color}
\usepackage{float}
\usepackage{multirow}
\usepackage{footmisc}
\usepackage{picinpar}
\usepackage{longtable,pdflscape}

\RequirePackage{xcolor}[2.11]
\colorlet{siaminlinkcolor}{red!50!black}
\colorlet{siamexlinkcolor}{blue!50!black}
\colorlet{siamreviewcolor}{black!50}

\usepackage[colorlinks,
            linkcolor=cyan,
            anchorcolor=green,
            citecolor=siaminlinkcolor
            ]{hyperref}
\usepackage{fancyhdr}
\usepackage{rotating}
\usepackage[noabbrev,capitalise,nameinlink]{cleveref}

\makeatletter
\@addtoreset{footnote}{section}
\usepackage{indentfirst}
\graphicspath{ {./Figures/} }

\begin{document}

\def\F{{\mathcal F}}
\def\K{{\mathcal T}}

\def\P{{\mathbb P}}
\def\N{{\mathbb N}}
\def\R{{\mathbb R}}
\def\T{{\mathbb T}}

\def\bfp{{\bf p}}
\def\bfq{{\bf q}}
\def\bfr{{\bf r}}
\def\bfg{{\bf g}}
\def\bfa{{\bf a}}
\def\bfb{{\bf b}}
\def\bfc{{\bf c}}
\def\bfd{{\bf d}}
\def\bfv{{\bf v}}
\def\bfx{{\bf x}}
\def\bfy{{\bf y}}
\def\bfz{{\bf z}}
\def\bfw{{\bf w}}
\def\bfu{{\bf u}}
\def\bfl{{\bf {\boldsymbol \lambda}}}
\def\bfe{{\bf 1}}
\def\supp{{\rm supp}}
\def\nhs{{\texttt{NHS}}}
\def\nhsa{{\texttt{NHST}}}

\newcommand{\qed}{\hphantom{.}\hfill $\Box$\medbreak}
\newcommand{\proof}{\noindent{\bf Proof \ }}
\newcommand{\sign}{{\mathrm{sgn}}}
\newcommand{\ts}{$\tau$-stationary point}
\newcommand{\tzk}{$\mathbb{T}_{\tau}(\bfz^k;s)$}
\newcommand{\beq}{\begin{eqnarray}}
\newcommand{\eeq}{\end{eqnarray}}
\newcommand{\erhao}{\fontsize{18pt}{\baselineskip}\selectfont}

\numberwithin{equation}{section}
\newtheorem{theorem}{Theorem}[section]
\newtheorem{lemma}{Lemma}[section]
\newtheorem{proposition}{Proposition}[section]
\newtheorem{corollary}{Corollary}[section]
\newtheorem{definition}{Definition}[section]
\newtheorem{remark}{Remark}[section]
\newtheorem{example}{Example}[section]
\newtheorem{assumption}{Assumption}[section]
\newenvironment{cproof}
{\begin{proof}
 [Proof.]
 \vspace{-3.2\parsep}}
{\renewcommand{\qed}{\hfill $\Diamond$} \end{proof}}

\renewcommand{\theequation}{
\arabic{equation}}
\renewcommand{\thefootnote}{\fnsymbol{footnote}}


\thispagestyle{empty}

\begin{center}
\topskip10mm
\erhao{\bf Heaviside Set Constrained Optimization:\\ Optimality and Newton Method}
\end{center}
\begin{center}
\setcounter{footnote}{0}
Shenglong Zhou, shenglong.zhou@soton.ac.uk\\
School of Mathematics, University of Southampton, UK\\
Lili Pan, panlili1979@163.com\\
Department of Mathematics, Shandong University of Technology,  China\\
Naihua Xiu, nhxiu@bjtu.edu.cn\\
 Department of Applied Mathematics, Beijing Jiaotong University,  China\\
{\small
}
\end{center}

\vskip12pt

\begin{abstract} Data in the real world  frequently involve  binary status: truth or falsehood, positiveness or negativeness, similarity or dissimilarity, spam or non-spam, and to name a few, with applications into the regression, classification problems and so on. To characterize the binary status, one of the ideal functions is the Heaviside step function that returns one for one status and zero for the other. Hence, it is of dis-continuity. Because of this, the conventional approaches to deal with the binary status tremendously benefit from its continuous surrogates. In this paper, we target the Heaviside step function directly and study the Heaviside set constrained optimization: calculating the tangent and normal cones of the feasible set, establishing several first-order sufficient and necessary optimality conditions, as well as developing a Newton type method that enjoys locally quadratic convergence and excellent numerical performance. 

\vspace{3mm}
 
\vspace{3mm}
\end{abstract}
{}
\numberwithin{equation}{section}

\section{Introduction}

In this paper, we study the  Heaviside set constrained optimization (HSCO):
\begin{equation}\label{ell-0-1}
\underset{\bfx\in\R^n}{\min}~ f(\bfx),~~{\rm s.t.}~ \|(A \bfx -\bfb)_+\|_0  \leq  s, 
\end{equation} 
where $f:\R^n\rightarrow\R$ is continuously differentiable, $A\in\R^{m\times n}$ with $m\leq n$,  $\bfb\in\R^{m}$ and $s < m $ is a given positive integer. Here ${\bfz}_+:=((z_1)_+, \cdots, (z_m)_+)^\top $ with $z_+:=\max\{z,{0}\}$. $\|\bfz\|_0$ is the so-called $\ell_0$ norm of $\bfz$, counting the number of its non-zero entries.  Therefore, $\|\bfz_+\|_0$ returns the number of positive entries in $\bfz$, and can be expressed as
\beq\label{heaviside}
\|\bfz_+\|_0= \frac{1}{2}\sum_{i=1}^m (\sign(z_i)+1),
\eeq
where $\sign(t)=1$ if $t$ is  positive and $-1$ otherwise. Here, we set $\sign(0)=-1$ instead of $0$ since it can be chosen as any scalar between 0 and 1  \cite{plan2012robust}. The function $h(t):= (\sign(t)+1)/2$ is known as the Heaviside step function (or the unit step function) from \cite{weisstein2002heaviside} named after Oliver Heaviside (1850-1925), an English mathematician and physicist. Therefore, we call the following set the Heaviside set: 
\begin{equation}
\label{l01-s}S:=\{\bfz\in\R^m:~\|\bfz_+\|_0\leq s\}.
\end{equation}
It is worth mentioning that the authors in  \cite{osuna1998reducing, evgeniou2000regularization, sollich2002bayesian} phrased $h(\cdot)$ the Heaviside step function, while the authors in \cite{friedman1997bias, LL2007, hastie2009elements, brooks2011support, ghanbari2019novel} called it the $0/1$ loss function. Motivations of \eqref{ell-0-1} include the support vector machine in marching learning, the 1-bit compressed sensing in signal processing, and to name a few. 
\subsection{Background}\label{application}

{\bf Example 1.1: Support vector machine (SVM).} It was first introduced by \cite{CV95} and then extensively applied into machine learning, pattern recognition, and to name a few. The task is to find a hyperplane in the input space that best separates the training data. More precisely, for the binary classification problem, suppose we are given a training data $\{(\bfa_i,c_i):i\in\N_m\}$ with $\bfa_i\in \R^{n}$ and $a_{in}=1$ being the samples and $c_i\in \{-1,1\}$ being the labels, where $\N_m:=\{1,2,\cdots,m\}$.  SVM aims at seeking a hyperplane $\langle\bfa,\bfx\rangle:=a_1x_1+\cdots+a_{n}x_{n}=0$ based on the training data, where $\bfx\in \R^{n}$ is the classifier to be trained.  
If the training data can be linearly separated in the input space, the unique optimal hyperplane can be obtained by solving the following convex quadratic programming problem:
\begin{eqnarray}\label{HM-SVM}  
\underset{\bfx\in \R^{n}}{\min}~  \|D\bfx\|^2,~~
  \mbox{s.t.}~c_{i} \langle \bfa_i, \bfx\rangle \geq1, \ i\in\N_m,
\end{eqnarray}
where $D$ is a diagonal matrix with $D_{ii}=1,i\in\N_{n-1}$, $D_{nn}\geq0$,  and $\|\cdot\|:=\|\cdot\|_2$ is the Euclidean norm. This  model is called hard-margin SVM because it requires all samples  classified correctly. However,  the training data are frequently  linearly inseparable, which means the above constraints cannot be fully satisfied. This scenario leads to the  soft-margin SVM model:
\begin{eqnarray}\label{SM-SVM}
\min_{\bfx\in \R^{n}}~~ \|D\bfx\|^2+ \mu\sum_{i=1}^{m}\ell(1-c_{i}\langle \bfa_i, \bfx\rangle),
\end{eqnarray}
where $\mu>0$ is a penalty parameter and $\ell$ can be some loss functions, such as the Hinge loss function $\ell(t)=t_+$ first introduced by \cite{CV95}. An impressive body of work has designed the loss functions: popular candidates include convex ones like the squared Hinge loss 
 in \cite{SV1999} and the pinball loss in \cite{HSS2014} and some non-convex ones in \cite{MBB2000,  CFWB2006,LL2007}. However, authors in \cite{ CV95,  LL2007, brooks2011support, nguyen2013algorithms, feng2016robust, wang2019support} pointed out that the ideal loss function is the $0/1$ loss function which turns out to be the Heaviside step function  \cite{osuna1998reducing, evgeniou2000regularization, sollich2002bayesian}, because it is completely robust to outliers and enables to attain the best learning-theoretic guarantee on predictive accuracy \cite{brooks2011support,nguyen2013algorithms,ustun2016supersparse}.  
 Since  the soft-margin SVM models allow  some samples to be misclassified,   if at most $s$ misclassified samples are allowable, we then could consider the following Heaviside set constrained model as a counterpart of the soft-margin SVM:
\begin{eqnarray}\label{HM-SVM-l01} 
\underset{\bfx\in \R^{n}}{\min}~ \|D\bfx\|^2,~~
\mbox{s.t.} ~\|( A\bfx +{\bf1})_+\|_0\leq s, 
\end{eqnarray}
where $A:=-[c_1\bfa_1~c_2\bfa_2~\cdots~c_m\bfa_m]^\top\in\R^{m\times n}$ and ${{\bf1}}:=[1~1~\cdots~1]^\top\in\R^m$.  We would like to emphasize that, compared with conventional soft-margin SVM  \eqref{SM-SVM}, the Heaviside set constrained model enjoys al least two advantages:
\begin{itemize}
\item[i)] It well captures the binary status of the problem: correctly and incorrectly classified samples. The Heaviside step function treats all  incorrectly classified samples equally, that is, counting 1 if a sample is misclassified (i.e., $t_i:=(A\bfx +{\bf1})_i>0$). However, the soft margin may treat misclassified samples unevenly. For example, the Hinge loss returns $t_i$  if  $t_i>0$. This means if one sample has an unexpectedly large $t_i$ (such a sample is called an outlier), then the model will be affected tremendously. Therefore,  the Heaviside set constrained model is more robust to the outliers than most  soft-margin SVM models.
\item[ii)] It is well known that tuning a proper penalty parameter $\mu>0$ in the soft-margin SVM   \eqref{SM-SVM} is always a tedious and tough task since the range of $\mu$ is usually unknown. A commonly used approach to select the parameter from a group of choices is the $K$-folder cross validation. But it may incur very expensive computational cost if $K$ is large or the choices are too many.  By contrast, $s$ in \eqref{HM-SVM-l01} is an integer from $\N_m$, so the range of $s$ is known. In some scenarios, the integer $s$ is able to be settled beforehand. For example, in regression or classification problems, $\|( A\bfx +{\bf1})_+\|_0/m$ calculates the ratio of the misclassified samples over the total samples, which is often required to be less than an acceptable tolerance, e.g., $5\%$. In such a case, one can set $s=\lceil 0.05m\rceil$, where $\lceil t\rceil$ returns the smallest integer that is no less than $t$.
\end{itemize}

{\bf Example 1.2: 1-bit compressed sensing (1-bit CS).} It was first introduced by \cite{BB08}. The basic idea is to reconstruct the signal $\bfx$ from the signs of the coded measurements, namely, $c_i=\sign(\langle\bfa_i,\bfx\rangle+\varepsilon_i), ~i\in\N_m$, where $\bfa_i\in\R^{ n}$ is the measurement, $c_i\in\{1,-1\}$ is the 1-bit measurement and  $\varepsilon_i$ is the noise.  Originally, the model takes the form:
\begin{eqnarray}\label{1-bit-cs}
\underset{\bfx\in \R^{n}}{\min}~\|\bfx\|_1,~~
\mbox{s.t.} ~ c_i \langle\bfa_i,\bfx\rangle \geq {0}, i\in\N_m,~ \|\bfx\|=1,
\end{eqnarray}
where $\|\bfx\|_1:=\sum_i|x_i|$ is the $\ell_1$ norm.  A common approach of dealing with these inequality  constraints in \eqref{1-bit-cs} is the same as that of processing \eqref{SM-SVM}, namely, penalizing them via some loss functions, such as 
$\sum(c_i-\langle\bfa_i,\bfx\rangle)^2$ in \cite{huang2018robust} and $\|(A\bfx  )_+\|^p_p$ with $p=1,2$ in \cite{laska2011trust,  jacques2013robust}, where $A$ is  defined  similarly in SVM. Particularly,   in \cite{dai2016noisy}, $\|(A\bfx+\epsilon {\bf1} )_+\|_0$ was benefited for quantifying the number of the incorrect signs, where $\epsilon$ is a given positive parameter to eliminate the zero solution.  Therefore, if the recovered signal is allowed to have $s$ incorrect signs,  we then could take the following Heaviside set constrained model into consideration:
 \begin{eqnarray}\label{1-bit-cs-l01}
\underset{\bfx\in \R^{n}}{\min}~ \phi(\bfx)+\eta\|\bfx\|^2,~~
\mbox{s.t.}~\| (A\bfx+\epsilon {\bf1})_+\|_0\leq s, 
\end{eqnarray} 
where $\phi$ is the function   pursuing  the sparse structure of a solution. For example,  the $\ell_1$ norm $\|\bfx\|_1$, which means the objective function in \eqref{1-bit-cs-l01} is the elastic net  \cite{zou2005regularization}, the smoothing $\ell_q$ norm $\sum_{i=1}^n(x_i^2+\varepsilon)^{q/2}$  with $0<q<1$ and $\varepsilon>0$ \cite{lai2013improved}, or the log penalty $\sum_{i=1}^n{\rm ln}(|x_i|+\varepsilon)$  with $\varepsilon>0$, which has a close relationship with the weighted $\ell_1$ norm \cite{candes2008enhancing}. Again, we would like to emphasize that the Heaviside set constrained model \eqref{1-bit-cs-l01} well captures the binary status of the data and thus is more robust to outliers,  in the meantime, the selection of the involved integer $s$ is easier  than that of $\mu$.

\subsection{Contributions} 
It is known that most conventional approaches to deal with the Heaviside step function involved optimization problems are based on surrogates. The relationships between the solutions to the surrogates and their original problems are unravelled, partially because of the hardness of establishing the optimality conditions of the original problems. However, this paper conquers the hardness and contributes to the following aspects.
\begin{itemize}
\item[1)] As far as we know, this is the first paper where the optimization problem \eqref{ell-0-1}, whose constraint function $\|(\cdot)_+\|_0$ is a combination of the sparsity and the Heaviside step function, is studied. It is motivated by at least two important real applications: the SVM and 1-bit CS problems. Compared with their conventional optimization models,  the Heaviside set constrained counterparts have their advantages. Moreover, it is known that the sparse set $\{\bfz\in\R^m:\|\bfz\|_0\leq s\}$  a union of finitely many subspaces, while the Heaviside set in \eqref{l01-s} is a union of finitely many polyhedral sets, which incurs tremendous difficulties. 
However, we succeed in the face of such difficulties.
\item[2)] To establish the optimality conditions of the problem  \eqref{ell-0-1}, we first investigate the properties of the   Heaviside set $S$ defined by \eqref{l01-s}, including calculating the projection of a given point onto $S$ by \cref{pro-1}, and deriving the Bouligand tangent cone and Fr$\acute{e}$chet normal cone  of $S$ by  \cref{tangent-S} and \cref{normal-S}. 
\item[3)] The feasible region of the problem \eqref{ell-0-1} turns out to be nonempty if $A$ has a row rank $m-s$ by  \cref{feasibility} for a given $s$.  To well understand the solutions to the problem, the normal cone of $S$ and the projection of $S$  respectively allow  us to define KKT points and $\tau$-stationary points to \eqref{ell-0-1}. We show that one of a local minimizer, a KKT point and a $\tau$-stationary point can be the other under some mild conditions. For example, a $\tau$-stationary point is a KKT point that is also a local minimizer if $f$ is convex. And under a mild condition, a local minimizer is a KKT point which is also a $\tau$-stationary point. Detailed relationships are summarized in  \cref{cor-relation}. 
\item[4)] A $\tau$-stationary point can be expressed by an equation system that enables us to take advantage of the Newton method. The proposed method is dubbed as {\tt NHS}, an abbreviation for the Newton method solving the Heaviside set constrained optimization.  It turns out to enjoy a locally quadratic convergence property under the standard assumptions, see \cref{the:quadratic}. However, when it comes to numerical computing, the involved parameter $s$ in \eqref{ell-0-1} is unknown beforehand in general. We thus integrate a tuning strategy of updating $s$ during the process in {\tt NHS}, which gives rise to  {\tt NHST} in  \cref{Alg-NL01-s}. Such a strategy makes  {\tt NHS} work relatively well  when it is benchmarked against several leading solvers for addressing the SVM and 1-bit CS problems. 
\end{itemize}

\subsection{Organization}
This paper is organized as follows. In the next section, we analyze the Heaviside set $S$, calculating the projection of one point onto it and deriving its  tangent cone and normal cone. In  \cref{sec:opt}, we first show the nonemptiness of the feasible region of the problem \eqref{ell-0-1}. Then based on the normal cone of $S$, we establish two kinds of optimality conditions of the problem \eqref{ell-0-1}: KKT points and $\tau$-stationary points, followed by the establishments of their relationships to the local or global minimizers. 
In \cref{sec:Newton}, Newton type method {\tt NHS} is designed to solve the $\tau$-stationary equations stemmed from the $\tau$-stationary points, and its locally quadratic convergence property is then achieved. In \cref{sec:numerical}, we develop an improved scheme of  {\tt NHS} (dubbed as {\tt NHST}), where a tuning strategy to select the unknown parameter $s$ is integrated, and then conduct extensive numerical experiments to demonstrate that {\tt NHST} is relatively competitive, when it is against a few leading solvers for addressing the SVM and 1-bit CS problems. Concluding remarks are given in the last section of this paper. 

\subsection{Notation} 
We end this section with defining some notation employed throughout this paper.  Give a subset $T\subseteq\N_m$, its cardinality and complementary set are $|T|$ and $\overline T$. For a vector $\bfz \in \R^{m}$, we define
\begin{eqnarray}\label{notation-z}
\begin{array}{lcllll}
|\bfz|&:=&(|z_1|,\cdots,|z_m|)^\top,  &\Gamma_+&:=&\left\{i\in\N_m:~ z_i>0  \right\},\\
\supp(\bfz)&:=&\{i\in \N_m: z_i\neq0\} , &\Gamma_0&:=&\left\{i\in\N_m:~ z_i=0 \right\},\\
U( \bfz ,\epsilon)&:=&\{\bfv\in\R^m: \|\bfv- \bfz \|<\epsilon\} ,~~ &\Gamma_-&:=&\left\{i\in\N_m:~ z_i<0 \right\}, 
\end{array}
\end{eqnarray}
where $\supp(\bfz)$ is the support of $\bfz$ and  $ U( \bfz ,\epsilon)$ is the neighbourhood of $ \bfz $ with a radius $\epsilon>0$. Note that $\Gamma_+, \Gamma_0$ and $\Gamma_-$ should depend  on $\bfz$. We drop their dependence if no extra explanations are provided since it would not cause confusion in the context.  Let $z_{[s]}$ be the $s$th largest element of $\bfz_+$, which means it is zero when $\|\bfz_+\|_0<s$ and positive otherwise. In addition, $\bfz_{T}$ (resp. $A_{T}$) represents the sub-vector (resp. sub-matrix) contains elements (resp. rows) of $\bfz$ indexed on $T$ and $(\bfx;\bfy):=(\bfx^\top~\bfy^\top)^\top$. For a scalar $z$, $\lceil z\rceil$ represents the smallest integer that is no less than $z$. The $i$th largest singular value of $H\in\R^{n\times n}$ is denoted by $\sigma_{i}(H)$, namely $\sigma_{1}(H)\geq\sigma_{2}(H)\geq\cdots\geq\sigma_{n}(H).$
Particularly, we write $\|H\|:=\sigma_{1}(H)$ (i.e., the spectral norm) and $\sigma_{\min}(H):=\sigma_{n}(H).$

\section{Properties of the Heaviside Set}
In this section, we pay our attention on the Heaviside set $S$ in \eqref{l01-s}, aiming at deriving the projection of a point onto it and its tangent and normal cones. To proceed that, we need the following notation that are used throughout the paper. Let $\Gamma_+, \Gamma_0$ and $\Gamma_-$ be given by \eqref{notation-z}. Denote
\begin{equation}
\label{T-z}
\T(\bfz;s):=\left\{T= (\Gamma_+\setminus\Gamma_{s})\cup \Gamma_0:  
\begin{array}{l}
\forall~\Gamma_{s}\subseteq  \Gamma_+, |\Gamma_{s}|=\min\{s, |\Gamma_+|\} \\
\forall~i\in \Gamma_{s},~z_i\geq z_j\geq0,~ \forall~j\in T 
\end{array}
\right\}.
\end{equation}
It is easy to see that for any $T\in\T(\bfz;s)$, 
\beq\label{z0+-Gamma+}  \overline T={\Gamma_{s} \cup  \Gamma_-}, ~~\|\bfz_+\|_0=|\Gamma_+|.\eeq Therefore,
if $\|\bfz_+\|_0\leq s$ then $\Gamma_{s}=\Gamma_+$ and hence $\T(\bfz;s)=\{\Gamma_0\}, \overline T=\Gamma_{+}\cup \Gamma_-$. If $\|\bfz_+\|_0 > s$, then $T\in \T(\bfz;s)$ captures the indices of the $|\Gamma_+|-|\Gamma_s|+|\Gamma_0|$ smallest non-negative entries of $\bfz$. Taking $\bfz=(3,2,2,0,-2)^\top$ for an instance, $\T(\bfz;3)=\{\{4\}\}$ and $\T(\bfz;2)=\{\{2,4\},\{3,4\}\}$.

\subsection{Projection}
For a nonempty and closed set $\Omega\subseteq\R^m$, the projection $\P_\Omega(\bfz)$ of $\bfz\in\R^m$ onto $\Omega$ is given by
\begin{eqnarray}
\P_\Omega(\bfz) = {\rm argmin} ~  \Big\{ \|\bfz-\bfu\|: \bfu\in\Omega \Big\}.
\end{eqnarray}
It is well known that the solution set of the right hand side problem is unique when $\Omega$ is convex and might have multiple elements otherwise in general.  The following property shows that the projection onto $S$ enjoys a the closed form.
\begin{proposition}\label{pro-1} Let $\T(\bfz;s)$ be defined by (\ref{T-z}). Then
\begin{equation}\label{psz}
\P_S(\bfz)= \Big\{\left(0;\bfz_{\overline{T}}\right): T\in\T(\bfz;s) \Big\}.
\end{equation} 
\end{proposition}
The proof is simple and thus is omitted here. Every point in the projection set is calculated by setting $\bfz_T=0$ for  $T\in\T(\bfz;s)$.   We  give an example to illustrate \eqref{psz}. Again consider the point $\bfz=(3,2,2,0,-2)^\top$. Then $\T(\bfz;3)=\{\{4\}\}, \P_S(\bfz)=\{\bfz\}$  and $\T(\bfz;2)=\{\{2,4\},\{3,4\}\}, \P_S(\bfz)=\{(3,0,2,0,-2)^\top,$ $ (3,2,0,0,-2)^\top\}$. Our next concept associated with the projection $\P_S(\bfz)$ is a fixed point inclusion.
\begin{proposition}\label{pro-eta} Given  $\bfl\in\R^m$ and $\tau >0$, a point $\bfy\in\R^m$  satisfies  
\begin{eqnarray}\label{uPu}
\bfy \in \P_S({\bfy}+\tau \bfl) 
\end{eqnarray}
if and only if it satisfies that
\begin{eqnarray}\label{uPu-equ}
\|\bfy_+\|_0\leq s,~~ \lambda_i 
\begin{cases}
=0,&i\in\supp(\bfy),\\
\in\left[0, y_{[s]}/\tau \right],&i\notin\supp(\bfy).
\end{cases}
\end{eqnarray}
\end{proposition}
\begin{proof} Direct verification can show that a point satisfying (\ref{uPu-equ}) also satisfies (\ref{uPu}). So we only prove `only if' part. It follows from \cref{pro-1} that
$$\bfy\in\P_S({\bfy}+\tau \bfl)=\left\{\left[
\begin{array}{c}
0\\
\bfy_{\overline T}+\tau\bfl_{\overline T}
\end{array}
\right]:~T\in\T(\bfy+\tau\bfl;s)\right\}
.$$
This derives $\|\bfy_+\|_0\leq s$, and for any $T\in\T(\bfy+\tau\bfl;s)$,
\beq\label{pro-z-T}\bfy_T=0,~~~\bfl_{\overline{T}}=0,~~~\bfy+\tau\bfl=\left[
\begin{array}{c}
\tau\bfl_T\\
\bfy_{\overline{T}} 
\end{array}\right],\eeq
which together with the definition of $\T(\bfy+\tau\bfl;s)$ in \eqref{T-z} gives rise to 
\beq\label{pro-supp-T}
 T =(\Gamma_+\setminus\Gamma_{s})\cup \Gamma_0,~~\overline T=\Gamma_s   \cup \Gamma_-=\supp(\bfy),\eeq
where $\Gamma_+, \Gamma_-$ and $\Gamma_0$ are defined as \eqref{notation-z} in which $\bfz$ is replaced by $\bfy+\tau\bfl$. \\
For $\|\bfy_+\|_0<s$, we must have $\bfl_{ {T}}=0$, leading to $\bfl=0$ and showing \eqref{uPu-equ}. If fact, suppose there is an $i\in T$ such that $\lambda_i\neq 0$. If $\lambda_i>0$, then $\bfy+\tau\bfl$ has at least $\|\bfy_+\|_0+1\leq s$ positive entries and thus $\|\bfy_+\|_0=\|(\P_S({\bfy}+\tau \bfl))_+\|_0 \geq \|\bfy_+\|_0+1$, a contradiction. If  $\lambda_i < 0$, then $y_i+\tau \lambda_i=\tau \lambda_i<0$ due to $i\in T$ and $\bfy_T=0$, which means $i\in\Gamma_- \subseteq \overline T$, another contradiction. \\
For $\|\bfy_+\|_0=s$, \eqref{uPu-equ} is satisfied  for any $j\in\supp(\bfy)= \overline T$ due to  $\bfl_{\overline T}=0$.   For $j\notin\supp(\bfy)$, namely, $j\in {T}$,   the definition of $\Gamma_s$ in \eqref{T-z} yields
\begin{eqnarray*}
\forall~j\in  T,~~  0\leq y_j+\tau  \lambda_j \leq  y_i +\tau  \lambda_i , ~~\forall~i\in  \Gamma_s,
\end{eqnarray*}
which together with  $\Gamma_s\subseteq \overline T$ and \eqref{pro-z-T} results in
\begin{eqnarray*}
\forall~j\in  T,~~  0\leq\tau  \lambda_j\leq y_i, ~~ \forall~i\in \Gamma_s.
\end{eqnarray*}
Hence,  $0\leq\tau \lambda_j\leq  {\min}_{i\in\Gamma_s}y_i= y_{[s]}, \forall~j\in  T (j\notin\supp(\bfy))$, showing \eqref{uPu-equ}. \qed
\end{proof}
\subsection{Tangent and Normal cones}
Recalling that for any nonempty set $\Omega\subseteq \mathbb{R}^{m}$, its  Bouligand tangent cone  $T_{\Omega}(\bfz)$  and corresponding  Fr$\acute{e}$chet normal cone  $\widehat{N}_{\Omega}(\bfz)$  at point $\bfz\in\Omega$ are defined as \cite{RW1998}:
\begin{eqnarray}\label{Tangent Cone}
T_{\Omega}(\bfz)&:=&  \left\{~\bfd\in\mathbb{R}^{m}:
 \begin{array}{r}
\exists~\eta_k\geq0,\{\bfz^k\}\subseteq\Omega, \underset{k\rightarrow \infty}{\lim}\bfz^k=\bfz\\
 \text{such~that}~\underset{k\rightarrow \infty}{\lim}\eta_k(\bfz^k-\bfz)=\bfd
\end{array}\right\},\\
\label{Normal Cone} \widehat{N}_{\Omega}(\bfz)&:=& \Big\{~\bfd\in\mathbb{R}^{m}:~\langle \bfd, \bfu\rangle\leq0,~\forall~\bfu\in T_{\Omega}(\bfz)~\Big\}.
\end{eqnarray}
To acquire these cones of the Heaviside set $S$ in \eqref{l01-s}, for a point $\bfz\in S$, we define the following index set 
\begin{eqnarray}\label{index-u}
\begin{array}{ccl}
{\mathcal P}&:=&\{  \Gamma \subseteq \Gamma_0:~ |\Gamma|\leq s- |\Gamma_+|\},
\end{array}
\end{eqnarray}
where  $\Gamma_0$ and $\Gamma_+$ are given by \eqref{notation-z}. One can discern that since $\|\bfz_+\|_0= |\Gamma_+|$ by \eqref{z0+-Gamma+},  ${\mathcal P}=\emptyset$ if $\|\bfz_+\|_0=s$ and  ${\mathcal P}\neq\emptyset$ if $\|\bfz_+\|_0<s$ and $\Gamma_0\neq \emptyset$. The definition of \eqref{index-u} allows for expressing the Bouligand tangent cone of $S$ explicitly by the following theorem. 
\begin{proposition}
  \label{tangent-S}
The  Bouligand tangent cone $T_{S}(\bfz)$    at  $\bfz\in S$ is given by
\begin{eqnarray}
\label{TS}T_{S}(\bfz)&=& \underset{\Gamma \in \mathcal{P}}{\bigcup}\Big\{\bfd\in\mathbb{R}^{m}: \bfd_{\Gamma_0\setminus \Gamma} \leq { 0 }\Big\}. 
\end{eqnarray}
\end{proposition}
\begin{proof}  Let $\Phi(\bfz)$ the set in right hand side of (\ref{TS}). We first verify the inclusion $T_{S}(\bfz)\subseteq \Phi(\bfz)$. Consider any $\bfd\in T_{S}(\bfz)$. Then it follows from (\ref{Tangent Cone}) that $\exists~\eta_k\geq0, \{\bfz^k\}\subseteq S,  \bfz^k\rightarrow\bfz$ such that $\eta_k(\bfz^k-\bfz)\rightarrow\bfd$. 
By $\bfz^k\rightarrow\bfz$, we have that for sufficiently large $k$,
 \begin{eqnarray}
&&\bfz^k_{\Gamma_{+}} \rightarrow \bfz _{\Gamma_{+}}>{ 0 } \Longrightarrow \bfz^k_{\Gamma_{+}} >{ 0 },\nonumber\\
\label{uku0}&&\bfz^k_{\Gamma_{-}} \rightarrow \bfz _{\Gamma_{-}}<{ 0 } \Longrightarrow\bfz^k_{\Gamma_{-}} <{ 0 }, \\
&&\bfz^k_{\Gamma_{0}} ~\rightarrow  \bfz_{\Gamma_{0}}={ 0 }.\nonumber
\end{eqnarray}
The index of each positive element of $\bfz^k$ is contained by  either  $\Gamma_{+}$ or $J_k:=\{i\in \Gamma_{0}: z_i^k>0\}$. It follows from $\|(\bfz^k)_+\|_0\leq s$ that $|J_k|\leq s-|\Gamma_{+}|$ and hence $J_k\in \mathcal{P}$. Since  $\Gamma_{0}$ has finitely many elements, $\{J_k\}$ is bounded and thus has a subsequence $J_{k} \equiv :J, \forall k\in\K$, where $\K$ is a subsequence of $\{1,2,3,\cdots\}$. What is more, we can conclude that for any $i\in \Gamma_{0} \setminus J$, we have $z_{i}^{k}\leq 0,k\in\K$. In fact, if there is an $i_0\in \Gamma_{0} \setminus J$ such that   $z^{t}_{i_0}> 0,t\in \K_1$, where $\K_1$ is a subsequence of $\K$. Then we let $J=J\cup\{i_0\}$ and consider the sub-subsequence $\{\bfz^{t},t\in \K_1\}$. 
So without loss of any generality, we focus on the subsequence  $\{\bfz^{k},k\in\K\}$ that satisfies
\begin{eqnarray}\label{uk<0}
z^{k}_j 
>0,~ j\in J,~~~~
z^{k}_i \leq0,~ i \in \Gamma_{0}\setminus J. 
\end{eqnarray}
Let $k\in\K$ and $k\rightarrow \infty$. The above conditions lead  to
\begin{eqnarray*}
\bfd_{\Gamma_{0} \setminus J}=\lim_{{k} {\rightarrow}\infty}\eta_{k}(\bfz^{k}-\bfz)_{\Gamma_{0} \setminus J}
\overset{(\ref{uku0})}{=}\lim_{{k} {\rightarrow}\infty}\eta_{k}
\bfz^{k}_{\Gamma_{0} \setminus J}
\overset{(\ref{uk<0})}{\leq}{ 0 }. \end{eqnarray*}
This and $J\in\mathcal{P}$ from $J_k\in \mathcal{P}$ show $\bfd\in \Phi(\bfz) $, which verifies $T_{S}(\bfz)\subseteq \Phi(\bfz)$.

Next we show $T_{S}(\bfz)\supseteq \Phi(\bfz)$. For any $\bfd\in \Phi(\bfz)$, there is  a $\Gamma\in\mathcal{P}$ such that $\bfd_{\Gamma_{0} \setminus\Gamma}\leq { 0 }$. Consider a positive sequence  $\eta_k\rightarrow 0$ and let $\bfz^k=\bfz+\eta_k \bfd $. These indicate  $\bfz^k\rightarrow\bfz$ and  $(\bfz^k-\bfz)/\eta_k=\bfd$. To check $\bfz^k\in S$, decompose $\bfz^k $ as
\begin{eqnarray*}
\bfz^k=\bfz+ \eta_k \bfd  = \left[\begin{array}{l}
(\bfz+ \eta_k\bfd  ) _{\Gamma_{-}}\\
(\bfz + \eta_k\bfd )_{\Gamma_{+}\cup \Gamma}\\
(\bfz + \eta_k\bfd )_{\Gamma_{0} \setminus\Gamma}\\
\end{array}\right]\overset{(\ref{uku0})}{=} \left[\begin{array}{l}
(\bfz+ \eta_k\bfd ) _{\Gamma_{-}}\\
(\bfz +\eta_k\bfd )_{\Gamma_{+}\cup \Gamma}\\
~~~~~~~ \eta_k  \bfd_{\Gamma_{0} \setminus\Gamma}    \\
\end{array}\right].
 \end{eqnarray*}
For sufficiently large $k$, it holds $\bfz_{\Gamma_{-}} + \eta_k\bfd _{\Gamma_{-}} \leq { 0 } $ by the definition of $\Gamma_-$ in \eqref{notation-z}, which together with $\eta_k\bfd_{\Gamma_{0} \setminus\Gamma}    \leq { 0 }$ suffices to 
\begin{eqnarray*}\|(\bfz^k)_+\|_0&=&\|((\bfz + \eta_k\bfd)_{\Gamma_{+}\cup \Gamma})_+\|_0\\
&\leq&|\Gamma_{+}| + | \Gamma|\leq |\Gamma_{+}|+s-|\Gamma_{+}|= s,\end{eqnarray*}
 where the last inequality is due to $\Gamma\in\mathcal{P}$.
Therefore, we have $\bfd\in T_{S}(\bfz)$, which verifies $T_{S}(\bfz)\supseteq \Phi(\bfz)$. This completes the whole proof. \qed
\end{proof}
The direct verification  allows us to derive the Fr$\acute{e}$chet normal cone of $S$ by (\ref{Normal Cone}). 
\begin{proposition}
  \label{normal-S}
The Fr$\acute{e}$chet normal cone $\widehat{N}_{S}(\bfz)$  at  $\bfz\in S$ can be expressed as
\begin{eqnarray}
\label{NS-exp}
\widehat{N}_{S}(\bfz)=
\begin{cases}
\left\{ \bfd\in\mathbb{R}^{m}:~  
 d_i\left\{
 \begin{array}{ll}
 =0&{\rm if}~~ z_i\neq 0\\
 \geq 0&{\rm if}~~ z_i= 0\\
 \end{array}\right.\right\},&\|\bfz_+\|_0=s,   \\
\hspace{3cm}\{ 0  \}, &\|\bfz_+\|_0<s.
\end{cases}
\end{eqnarray}
\end{proposition}
\begin{proof} If $\|\bfz_+\|_0=s$, then  $|\Gamma_{+}|=s$, leading to  ${\mathcal P}=\emptyset$. This together with \eqref{TS} yields $T_{S}(\bfz)=\left\{\bfd\in\mathbb{R}^{m}: \bfd_{\Gamma_0} \leq { 0 }\right\} $ and hence derives \eqref{NS-exp} by the definition \eqref{Normal Cone} of $\widehat{N}_{S}(\bfz)$. Now consider the case  $\|\bfz_+\|_0<s$. If $\Gamma_0=\emptyset$, then  $T_{S}(\bfz)=\left\{\bfd\in\mathbb{R}^{m}: \bfd_{\emptyset} \leq { 0 }\right\} $ and hence gives rise to \eqref{NS-exp}.
If $| \Gamma_0| \geq  1$, then  for any $i\in \Gamma_0$, $\{i\}\in \mathcal P$ because of $|\{i\}|=1\leq s- \|\bfz_+\|_0=s-|\Gamma_+|$. This means for any $\bfd\in T_S(\bfz)$, it holds $\bfd_{\Gamma_0\setminus \{i\}}\leq 0$ and $d_j\in\R, j\in\{i\}\cup\overline{\Gamma}_0$, as a result, $\bfd'_{\{i\}\cup\overline{\Gamma}_0}=0$  for any $\bfd'\in \widehat{N}_{S}(\bfz)$. The arbitrariness of $i$ taken  from $\Gamma_0$ suffices to $\bfd'=0$.\qed
\end{proof}
We end this section with giving one example to illustrate the  tangent and normal cones of $S$ at three points: $\bfz_1=(0,1)^\top$, $\bfz_2=(-1,0)^\top$ and $\bfz_3=(0,0)^\top$ in two dimensional space,  namely, $$S=\{\bfz\in\R^2:~\|\bfz_+\|_0\leq 1\}=\{\bfz\in\R^2:~z_1\leq 0~{\rm or}~z_2\leq 0\}.$$ 
 It can be easily seen that 
\begin{eqnarray*}\begin{array}{llllll}
T_{S}(\bfz_1)&=& \{\bfd\in\mathbb{R}^{2}:d_1 \leq 0 \},~~&
\widehat{N}_{S}(\bfz_1) &=&  \{ \bfd\in\mathbb{R}^{2}:d_1\geq0,d_2=0 \},\\
T_{S}(\bfz_2)&=&\mathbb{R}^{2},&
\widehat{N}_{S}(\bfz_2)&=&\left\{ { 0 } \right\},\\
T_{S}(\bfz_3)&=&S,&\widehat{N}_{S}(\bfz_3)&=&\left\{ { 0 } \right\}.
\end{array}
\end{eqnarray*}

\section{Optimality Conditions} \label{sec:opt}
The first issue we encounter is the feasibility  of the problem (\ref{ell-0-1}), which is guaranteed by the following theorem, where the assumption can be guaranteed if $A$ has a row rank $m-s$.
\begin{theorem}\label{feasibility} For any given $s\in\N_m$, the problem (\ref{ell-0-1}) is feasible if  there is a $T\subseteq \N_m$ with $|T|\geq m-s$ such that $A_T$ is full row rank.
\end{theorem}
\begin{proof} By assumption, there is a $T\subseteq \N_m$ with $|T|\geq m-s$ such that $A_T$ is full row rank. If the solution set of $A_T \bfx \leq\bfb_T$ is nonempty, then we have 
 \begin{eqnarray*}\|(A \bfx -\bfb)_+\|_0 &=&\|(A_T \bfx -\bfb_T)_+\|_0 +\|(A_{\overline{T}} \bfx -\bfb_{\overline{T}})_+\|_0\\
 & =&\|(A_{\overline{T}} \bfx -\bfb_{\overline{T}})_+\|_0\leq m-|T| \leq  s.\end{eqnarray*}
 Therefore, we need to prove that the solution set of $A_T \bfx \leq\bfb_T$ is nonempty. It follows from \cite[Theorem 7]{perng2017class} or \cite{gale1989theory} that $A_T \bfx \leq\bfb_T$ admits a solution if and only if the following system has no solution,
 $$A_T^\top \bfu =0,~~ \bfu\geq0,~~ \langle \bfu,~\bfb_T \rangle <0.$$
Apparently, $\bfu=0$ from the first equation since $A_T$ is full row rank, causing a contradiction  $0=\langle \bfu,~\bfb_T \rangle <0$, which shows the desired result.\qed
\end{proof}
Next, we establish  the first order necessary and sufficient optimality conditions of  (\ref{ell-0-1}). To proceed this, we  consider an  equivalent formulation of the problem (\ref{ell-0-1}),
\begin{eqnarray}
\label{ell-0-2}
\underset{\bfx,\bfy}{\min}&&  f(\bfx),\\
{\rm s.t.}&& \bfy=A \bfx -\bfb,~ \| \bfy_+\|_0  \leq  s.\nonumber 
\end{eqnarray}
Before the main theorems ahead of us, we first denote the feasible sets by
 \begin{eqnarray}
\label{S1} \F&:=&\Big\{(\bfx, \bfy)\in\R^n\times\R^m:~\bfy = A \bfx-\bfb,\| \bfy_+\|_0 \leq  s \Big\}, \\
\F_1&:=&\Big\{(\bfx, \bfy)\in\R^n\times\R^m:~\bfy = A \bfx-\bfb\Big\}.\nonumber
 \end{eqnarray}
 Given a point $(\bfx^*,\bfy^*)\in\F$ and a constant $\delta>0$, denote 
  \begin{eqnarray} 
\label{delta*}   
\delta_* &:=&
 \left\{\begin{array}{ll}
 \min \Big\{\delta, \sqrt{2m}~\underset{i}{\min}~ \{|y_i^*|: y_i^*\neq0\} \Big\},&~~\bfy^*\neq{ 0 }, \\
 \delta&~~\bfy^*={ 0 }, 
 \end{array}
 \right.\\
J_*&:=& \{i\in\N_m: (A\bfx^*-\bfb)_i=0\}, ~~~~
\delta_m~:=~ {\delta_* }/{\sqrt{2m}}. 
  \nonumber
 \end{eqnarray}
 Based on above constants, we also define a local region  of $(\bfx^*,\bfy^*)$ by
  \begin{eqnarray}\label{F-*}
\F_* &:=&
 \Big\{(\bfx, \bfy)\in\R^n\times\R^m:
 \|\bfx-\bfx^*\|<\delta_* /\sqrt{2}, ~\bfy\in S_*
\Big\}, \end{eqnarray}
where $S_*$ is given by
  \begin{eqnarray}\label{S-*}
  S_* &:=&\left\{\begin{array}{lll}
  ~~~~~~\left\{
 \bfy\in\R^m:
 \begin{array}{rccl}
 |y_i -y_i^*|&\leq& \delta_m,&i\in \overline{J}_* \\
  -\delta_m~\leq~y_i &\leq&  0,&i\in J_* 
  \end{array}
 \right\}, & \|\bfy^*_+\|_0=s,\\
 ~\\
 \underset{j\in J_*}{ \bigcup}\left\{
 \bfy\in\R^m:
 \begin{array}{rccl}
 |y_i -y_i^*|&\leq& \delta_m,&i\in \overline{J}_*\cup\{j\}\\
  -\delta_m~\leq~ y_i &\leq& 0,&i\in J_*\setminus\{j\}
  \end{array}
 \right\}, & \|\bfy^*_+\|_0<s. 
 \end{array}\right.
 \end{eqnarray}
 Some properties of above sets are given as below.
\begin{lemma} \label{lemma:F123}
 Consider a point $(\bfx^*,\bfy^*) \in \F$.
  The following properties hold.
 \begin{itemize}
 \item[a)] $S_*\subseteq S$ and $\widehat N_{S_*}(\bfy^*)=\widehat{N}_{S}(\bfy^*)$.
  \item[b)] $\F_*  \subseteq U(\bfx^*, \bfy^*,\delta_*)$.
  \item[c)]$\widehat N_{\F_1\cap \F_*}(\bfx^*, \bfy^*)=\widehat N_{\F_1}(\bfx^*, \bfy^*)+\widehat N_{\F_*}(\bfx^*, \bfy^*)$ holds if $A_{J_*}$ is  full row rank.
 \end{itemize}
 \end{lemma}
\begin{proof}  a)  Since $(\bfx^*,\bfy^*) \in \F$,  from \eqref{delta*}, we have $$J_*=\{i\in\N_m: (A\bfx^*-\bfb)_i=0\}=\{i\in\N_m: y^*_i=0\}.$$    
If $\|\bfy^*_+\|_0=s$, then for any $\bfy\in S_*$ and any $i\in \overline{J}_*$, it follows
 \begin{eqnarray}
 \label{uup}y_i^*>0,~ |y_i -y_i^*|\leq\delta_m &~~\Longrightarrow~~& y_i \geq~y_i^*- \delta_m\overset{\eqref{delta*}}{\geq} ~y_i^*- \underset{i\in \overline{J}_*}{\min} |y_i^*| \geq 0,~~\\
 \label{uun}y_i^*<0,~ |y_i -y_i^*|\leq \delta_m&~~\Longrightarrow~~& y_i \leq~ y_i^*+\delta_m\overset{\eqref{delta*}}{\leq} ~y_i^*+ \underset{i\in \overline{J}_*}{\min} |y_i^*| \leq 0.~~
 \end{eqnarray}
 These mean $\bfy$ has at most $\|\bfy^*_+\|_0$ positive elements, therefore, $\bfy\in S$ and $S_*  \subseteq S$. By the convexity of $S_*$, one can easily to derive that
  \begin{eqnarray}\label{normal-S2}
\widehat N_{S_*}(\bfy^*)&=& \left\{\bfd\in\R^m:~d_i\left\{
 \begin{array}{ll}
 = 0, &~ i\in \overline{J}_*\\
 \geq 0, &~ i\in J_*\\
 \end{array}
 \right.\right\}
\overset{\eqref{NS-exp}}{=}\widehat{N}_{S}(\bfy^*). 
 \end{eqnarray}
If $\|\bfy^*_+\|_0<s$, by denoting $S_{j}$ as
$$ S_{j}:=\left\{\bfy\in\R^m:\begin{array}{ll}
 |y_i -y_i^*|\leq \delta_m,&i\in \overline{J}_*\cup\{j\}\\
  -\delta_m\leq y_i \leq 0,&i\in J_* \setminus\{j\}
  \end{array}
 \right\},~~{\rm then}~~S_*={\bigcup}_{j\in J_* }S_{j}$$
by \eqref{S-*}. For any $\bfy\in S_*$, there is one $j\in J_*$  such that $\bfy\in S_{j}$. The same reasoning to show \eqref{uup} and \eqref{uun} is able to show that $\bfy$ has at most $\|\bfy^*_+\|_0+1\leq s$ positive elements. Therefore, $\bfy\in S$ and  $S_*  \subseteq S$. By the convexity of $S_j$, it follows
 \begin{eqnarray}
\label{normal-S2-1-0}  \widehat N_{S_{j}}(\bfy^*)  
 &= & \left\{\bfd\in\R^m:~d_i\left\{
 \begin{array}{ll}
 = 0, &~ i\in \overline{J}_*\cup\{j\}\\
 \geq 0, &~ i\in J_*\setminus \{j\}\\
 \end{array}
 \right.\right\}. 
 \end{eqnarray}
In addition, for a group of convex sets $\Omega_{i}$, it follows from  \cite[Proposition 3.1]{ban2011lipschitzian} that  
 \begin{eqnarray}\label{cond-pro31}
  \widehat N_{\underset{i }{\bigcup} \Omega_{i}}  (\bfx) 
= \bigcap_{i\in \{i:~\bfx\in \Omega_{i}\}} \widehat N_{\Omega_{i}}(\bfx).
 \end{eqnarray}
Therefore, $S_*={\bigcup}_{j\in J_* }S_{j}$ with each $S_{j} $ being convex and $\bfy^*\in S_{j}, \forall~j\in J_*$ yield
 \begin{eqnarray}
\widehat N_{S_*}(\bfy^*)
& 
\overset{\eqref{cond-pro31}}{=} &\bigcap_{j\in J_*} \widehat N_{S_{j}}(\bfy^*) \nonumber\\
\label{normal-S2-1} &\overset{\eqref{normal-S2-1-0}}{=} &\{0\}\overset{\eqref{NS-exp}}{=}\widehat{N}_{S}(\bfy^*).
 \end{eqnarray}
 b) For any $(\bfx, \bfy)\in \F_*$, where $\bfy\in S_*={\bigcup}_{j\in J_*}S_{j}$,  there exists $j\in J_*$  such that $\bfy\in S_{j}$.  Therefore, $\F_* \subseteq U(\bfx^*, \bfy^*,\delta_*)$ holds immediately owing to
  \begin{eqnarray*}
\|\bfx-\bfx^*\|^2+\| \bfy-\bfy^*\|^2&<&\frac{\delta_*^2}{2}+
\| (\bfy-\bfy^*)_{\overline{J}_*\cup\{j\}}\|^2+\| (\bfy-\bfy^*)_{J_* \setminus\{j\}}\|^2
\\&<&\frac{\delta_*^2}{2}  +\frac{\delta_*^2}{2m}\Big[|\overline{J}_*|+1\Big]+\frac{\delta_*^2}{2m}\Big[ |J_* |-1\Big]=\delta^2_*.\end{eqnarray*}
c) Let $X := \{\bfx\in\R^n:
 \|\bfx-\bfx^*\|<\delta_* /\sqrt{2}\}$. We prove the conclusion by two cases.  {\bf Case I} $\|\bfy^*_+\|_0<s$. Obviously, $\F_1\cap(X \times S_{j}) $ is convex and  $(\bfx^*, \bfy^*)\in\F_1\cap(X \times S_{j}) $. It has
 \begin{eqnarray*}
\F_1\cap \F_* &=& \F_1 \cap \left[ X \times S_* \right]= \F_1 \cap  \Big[ X \times \underset{j\in J_* }{\bigcup}S_{j}  \Big]\\
 &=&  \F_1 \cap  \Big[\underset{j\in J_* }{\bigcup}   X \times S_{j}  \Big]=
  \underset{j\in J_*}{\bigcup} \Big[\F_1\cap(X \times S_{j} )\Big]. 
\end{eqnarray*}
These three facts results in
 \begin{eqnarray}\label{N-F1-F3}
\widehat N_{\F_1\cap \F_*}(\bfx^*, \bfy^*)\overset{\eqref{cond-pro31}}{=}
 \underset{j\in J_*}{\bigcap} \widehat N_{\F_1\cap(X \times S_{j} ) }(\bfx^*, \bfy^*).
\end{eqnarray}
Next, we prove that
\begin{eqnarray} \label{N-X-Fi}
\widehat N_{\F_1\cap(X \times S_{j} ) }(\bfx^*, \bfy^*)=\widehat N_{\F_1}(\bfx^*, \bfy^*)+\widehat N_{X \times S_{j}}(\bfx^*, \bfy^*).
\end{eqnarray}
 In fact,   \cite[Proposition 2.12]{MN} states that  \begin{eqnarray}\label{normal-F1}
\widehat N_{\F_1}(\bfx^*,\bfy^*)=\left\{ (A^\top \bfl, -\bfl)\in\R^n\times\R^m: \bfl\in \R^m\right\} \end{eqnarray} 
and  \cite[Theorem 6.41]{RW1998} states that
\begin{eqnarray}\label{normal-F3}
\widehat N_{X \times S_{j}}(\bfx^*,\bfy^*)=\{{ 0 }\}\times \widehat N_{S_{j}}(\bfy^*).
\end{eqnarray} 
Note that  $\F_1$ and $ X \times S_{j}$ are convex and thus are regular at $\bfx^*, \bfy^*$ regarding  \cite[Definition 6.4]{RW1998}.  Therefore, to check \eqref{N-X-Fi},  \cite[Theorem 6.42]{RW1998} indicates that we only need to check the following inclusion  
   \begin{eqnarray}\label{inclusion-lam-d-0}(A^\top \bfl,  -\bfl)+( { 0 }, \bfd)={ 0 }~~\Longrightarrow~~ (A^\top \bfl,  -\bfl)=( { 0 }, \bfd)={ 0 } \end{eqnarray}
holds for any $ (A^\top \bfl,  -\bfl)\in \widehat N_{\F_1}(\bfx^*, \bfy^*)$ and $ ( { 0 }, \bfd)\in \widehat N_{X \times S_{j}}(\bfx^*, \bfy^*)$ with $\bfd \in \widehat N_{ S_{j}}( \bfy^*)$.  The condition in \eqref{inclusion-lam-d-0} delivers 
$\bfl_{ \overline{J}_*}=-\bfd_{ \overline{J}_*}=0$ by \eqref{normal-S2-1-0} and hence
\begin{eqnarray} 
0\overset{\eqref{inclusion-lam-d-0}}{=}A^\top \bfl&=&A^\top _{J_*}\bfl_{J_*}+A^\top _{\overline J_*}\bfl_{\overline J_*}=A^\top _{J_*}\bfl_{J_*}. 
\end{eqnarray}
As a result, $\bfl_{J_*}=0$ since $A_{J_*}$ is full row rank,  showing $\bfl=\bfd=0$ and thus \eqref{inclusion-lam-d-0}. Namely,   \eqref{N-X-Fi} is true, which contributes to
 \begin{eqnarray} \label{N-PHI}
\widehat N_{\F_1\cap \F_*}(\bfx^*, \bfy^*)  
 & \overset{\eqref{N-F1-F3},\eqref{N-X-Fi}}{=} &
 \underset{j\in J_* }{\bigcap}\Big[\underset{\Phi_{j}}{\underbrace{\widehat N_{\F_1}(\bfx^*, \bfy^*)+\widehat N_{X \times S_{j}}(\bfx^*, \bfy^*)}} \Big].
\end{eqnarray} 
We now conclude that 
\begin{eqnarray}\label{NF1F3}
\widehat N_{\F_1\cap \F_*}(\bfx^*, \bfy^*) = \widehat N_{\F_1}(\bfx^*, \bfy^*).
\end{eqnarray}
To show this, we first verify $ \widehat N_{\F_1}(\bfx^*, \bfy^*) \subseteq \widehat N_{\F_1\cap \F_*}(\bfx^*, \bfy^*)$. This is clearly true because of $0\in \widehat N_{X \times S_{j}}(\bfx^*, \bfy^*)$ for any $j\in J_*$ and
$$\widehat N_{\F_1}(\bfx^*, \bfy^*)=  \underset{j\in J_*}\bigcap \Big[\widehat N_{\F_1}(\bfx^*, \bfy^*)+\{0\}\Big]\subseteq \underset{j\in J_*}\bigcap \Phi_{j}\overset{\eqref{N-PHI}}{=}\widehat N_{\F_1\cap \F_*}(\bfx^*, \bfy^*).$$
We next verify $\widehat N_{\F_1\cap \F_*}(\bfx^*, \bfy^*) \subseteq \widehat N_{\F_1}(\bfx^*, \bfy^*)$. 	For any  point $\bfv\in \widehat N_{\F_1\cap \F_*}(\bfx^*, \bfy^*),$ it follows from \eqref{N-PHI} that $\bfv\in \Phi_{i} \cap \Phi_{j}$ for any $i,j \in J_*$ with $i \neq j$. This means there exist  \begin{eqnarray*} \bfl^1\in \widehat N_{\F_1}(\bfx^*, \bfy^*),&& (0;\bfd^1)\in \widehat N_{X \times S_{i}}(\bfx^*, \bfy^*)\\
\bfl^2\in \widehat N_{\F_1}(\bfx^*, \bfy^*),&& (0;\bfd^2)\in \widehat N_{X \times S_{j}}(\bfx^*, \bfy^*)\end{eqnarray*} 
such that $\bfd^1_{\overline{J}_* \cup \{i\}}=\bfd^2_{\overline{J}_* \cup \{j\}}=0$ by \eqref{normal-S2-1-0} and
\begin{eqnarray} \label{v-d-d}
\bfv&=&\left[ 
A^\top \bfl^1,~ \bfl^1 \right]+\left[ 0,~\left[0;~\bfd^1_{J_*\setminus \{i\}}\right]\right]
=\left[ A^\top \bfl^1,~\bfl^1 +\left[0;~\bfd^1_{J_*\setminus \{i\}}\right]\right]\nonumber\\
&=&\left[ A^\top \bfl^2,~\bfl^2 \right]+\left[ 0,~  \left[ 0;~\bfd^2_{J_*\setminus \{j\}}\right]\right]
=\left[ A^\top \bfl^2,~\bfl^2 +\left[0;~\bfd^2_{J_*\setminus \{j\}}\right]\right],~~~~ 
\end{eqnarray} 
which results in $A^\top\bfl^1 =A^\top\bfl^2$ and $\bfl^1_{\overline{J}_*}=\bfl^2_{\overline{J}_*}$. Consequently, we have
\begin{eqnarray*} 0=A^\top(\bfl^1 -  \bfl^2 )&=&A^\top_{J_*}(\bfl^1_{J_*} -  \bfl^2_{J_*} )+A^\top_{\overline{J}_*}(\bfl^1_{\overline{J}_*} -  \bfl^2 _{\overline{J}_*})\\
&=&A^\top_{J_*}(\bfl^1_{J_*} -  \bfl^2 _{J_*}),\end{eqnarray*} 
and thus $\bfl^1_{J_*} -  \bfl^2 _{J_*}=0$ due to the full row rankness of $A_{J_*}$.  Namely, $\bfl^1=  \bfl^2$, which by \eqref{v-d-d} enables us to prove $\bfd^1_{J_*\setminus \{i\}}=\bfd^2_{J_*\setminus \{i\}}$, sufficing to $\bfd^1=  \bfd^2$.
And thus $\bfd^1_{\overline{J}_*\cup \{i,j\}} =0$ due to  $\bfd^2_{\overline{J}_*\cup \{j\}}=0$. Since $j (\neq i)$ is chosen arbitrarily from ${J_*}$, we can claim that $\bfd^1=\bfd^1_{\overline{J}_*\cup {J_*}}=0$ and thus $\bfv=(
A^\top \bfl^1,
\bfl^1)\in \widehat N_{\F_1}(\bfx^*, \bfy^*)$. Overall, we show \eqref{NF1F3}, which allows us to make the conclusion by  
\begin{eqnarray*}
\widehat N_{\F_1\cap \F_*}(\bfx^*, \bfy^*)&\overset{ \eqref{NF1F3}}{=}& \widehat N_{\F_1}(\bfx^*, \bfy^*)\\
&=&\widehat N_{\F_1}(\bfx^*, \bfy^*)+\{0\}\times \{0\}\\
&\overset{ \eqref{normal-S2-1}}{=}&
\widehat N_{\F_1}(\bfx^*, \bfy^*)+ \widehat N_{X}(\bfx^*)\times \widehat N_{S_*}(\bfy^*) \\
&=&\widehat N_{\F_1}(\bfx^*, \bfy^*)+  \widehat N_{\F_*}(\bfx^*,\bfy^*).
\end{eqnarray*}
{\bf Case II}  $\| \bfy^*_+\|_0=s$. Same reasoning to prove \eqref{N-X-Fi} is able to show
\begin{eqnarray*} 
\widehat N_{\F_1\cap(X \times S_* ) }(\bfx^*, \bfy^*)=\widehat N_{\F_1}(\bfx^*, \bfy^*)+\widehat N_{X \times S_*}(\bfx^*, \bfy^*).
\end{eqnarray*}
Then the conclusion can be made immediately due to
$\F_* = X \times\F_2$. \qed
 \end{proof}
\subsection{KKT points}
We call  $\bfx^*\in\R^n$  is a KKT point of  (\ref{ell-0-1})  if there is a $\bfl^*\in\R^m$ such that
  \begin{eqnarray}\label{KKT-point}
  \begin{cases}
  \nabla f(\bfx^*)+A^\top\bfl^*&=~~{0},\\
 ~~~\widehat{N}_S(A\bfx^*- \bfb)&\ni~~\bfl^*,\\
 ~~\|(A\bfx^*- \bfb)_+\|_0&\leq~~ s.
 \end{cases}
  \end{eqnarray}
From now on, again as in \eqref{delta*}, we always denote
\beq\label{gamma*}  J_*:=\{i\in\N_m:~(A\bfx^*-\bfb)_i = 0\}.\eeq
We first build the relation between a KKT point and a local minimizer. 
\begin{theorem}[KKT points and local minimizers]\label{nec-suff-opt-con-KKT} The following relationships hold for the problem (\ref{ell-0-1}). 
\begin{itemize}
\item[a)] A local minimizer $\bfx^*$ is a KKT point if $A_{J_*}$ is full row rank. Furthermore, if $\|(A\bfx^*-\bfb)_+\|_0<s$, then $\bfl^*=0$ and $\nabla f(\bfx^*)=0$. 
\item[b)] Suppose $f$ is convex. A KKT point is  a local minimizer   if $\|(A\bfx^*-\bfb)_+\|_0=s$ and a global minimizer if $\|(A\bfx^*-\bfb)_+\|_0<s$.
\end{itemize}
\end{theorem}
\begin{proof}  a) Consider a local  minimizer $\bfx^*$  of (\ref{ell-0-1}). Then $(\bfx^*,\bfy^*)$ with $ \bfy^* = A\bfx^*- \bfb$  is also a  local  minimizer of (\ref{ell-0-2}). 
So we have  $
 \|\bfy^*_+\|_0 \leq  s$ and there is a $\delta>0$ such that $(\bfx^*,\bfy^*)$ is a global minimizer of the following problem
\begin{eqnarray}
\label{ell-0-1-1}
\underset{\bfx,\bfy}{\min}~ f(\bfx),~~
{\rm s.t.}~   &&\bfy = A \bfx-\bfb, ~  \bfy\in S,\\
&&  (\bfx, \bfy)\in  U((\bfx^*, \bfy^*), \delta). \nonumber
\end{eqnarray}
 Now consider an even much smaller radius $\delta_*$ by \eqref{delta*}. Then, $(\bfx^*, \bfy^*)$  is a minimizer of the problem
\begin{eqnarray}
\label{ell-0-1-2}
\underset{\bfx,\bfy}{\min}~ f(\bfx),~~
{\rm s.t.}~   &&\bfy = A \bfx-\bfb, ~   \bfy\in S,\\
&&   (\bfx, \bfy)\in  U((\bfx^*, \bfy^*), \delta_*). \nonumber
\end{eqnarray}
Let $\F_1,  \F_*$ and $S_*$ be defined by \eqref{S1}, \eqref{F-*} and \eqref{S-*}. These and \cref{lemma:F123} result in
  $\F_*  \subseteq U(\bfx^*, \bfy^*,\delta_*)$   and   $\bfy\in S_*  \subseteq S$ for any $(\bfx, \bfy) \in \F_*$.
In addition,  $(\bfx^*, \bfy^*)\in \F_1\cap \F_* $. Hence, (\ref{ell-0-1-2}) allows us to conclude that $(\bfx^*, \bfy^*)$ is also a minimizer of the following problem
\begin{eqnarray*}
&&\underset{\bfx, \bfy}{\min}~~f(\bfx), ~~~~ {\rm s.t.}~~(\bfx, \bfy)\in \F_1\cap \F_*.
\end{eqnarray*} 
Based on   \cite[Theorem 6.12]{RW1998}, a  minimizer of  the above problem satisfies 
\begin{eqnarray*}\label{normal-D}-(\nabla f(\bfx^*),0)&\in& \widehat N_{\F_1\cap \F_*}(\bfx^*, \bfy^*)\\
&\overset{  \cref{lemma:F123} }{=}&\widehat N_{\F_1}(\bfx^*, \bfy^*)+\widehat N_{\F_*}(\bfx^*, \bfy^*)\\
&\overset{(\ref{normal-F1}, \ref{normal-F3}) }{=}&\{ (A^\top \bfl, -\bfl)\in\R^n\times\R^m: \bfl\in \R^m\}+\{{ 0 }\}\times \widehat N_{S_*}(\bfy^*)\\
&\overset{  \cref{lemma:F123} }{=}&\{ (A^\top \bfl, -\bfl)\in\R^n\times\R^m: \bfl\in \R^m\}+\{{ 0 }\}\times \widehat N_{S}(\bfy^*). \end{eqnarray*}
As a result, there is a $\bfl^*\in\R^m$ such that  
$$\nabla f(\bfx^*)+ A^\top\bfl^*={0},~~~~\bfl^*\in \widehat{N}_S(\bfy^*).$$
These together with $ \bfy^* = A\bfx^*- \bfb$  and  $
 \|\bfy^*_+\|_0 \leq  s$ show \eqref{KKT-point}. Furthermore, if $
 \|\bfy^*_+\|_0 <  s$, then $\widehat N_{S}(\bfy^*)=\{0\}$ from \eqref{NS-exp}, yielding $\bfl^*=0$ and $\nabla f(\bfx^*)={0}$.\\
 
b) Let $(\bfx^*, \bfl^*)$ satisfy (\ref{KKT-point}) and  $ \bfy^* = A\bfx^*- \bfb$. If $\|\bfy^*_+\|_0 <  s$, then \eqref{KKT-point} and $\widehat N_{S}(\bfy^*)=\{0\}$ from \eqref{NS-exp} suffice to $\bfl^*=0$ and $\nabla f(\bfx^*)={0}$. It follows from the convexity of $f$   that 
\begin{eqnarray*} 
f(\bfx)  
&\geq&f( \bfx^*)+\langle \nabla f( \bfx^*),\bfx-\bfx^*\rangle =  f( \bfx^*). 
\end{eqnarray*}
This displays the global optimality of $\bfx^*$.  

Now we focus on $\|\bfy^*_+\|_0 = s$.    Consider a local region $U((\bfx^*, \bfy^*), \delta_0)$ with $\delta_0:=\min_i \{ y_i^*: y_i^*>0\}$. For any $ (\bfx, \bfy)\in \F \cap U((\bfx^*, \bfy^*), \delta_0)$ where $\F$ is defined by \eqref{S1}, we claim three facts: F1)
\begin{eqnarray}\label{yyAxx}
\bfy-\bfy^*=A(\bfx-\bfx^*).
\end{eqnarray}
F2) Since  $\bfy^* = A \bfx^*-\bfb$ and the definition of $J_*$ in \eqref{gamma*} that  $\bfy^*_{J_* }=0$ and  $\bfy^*_{\overline{J}_* }\neq  0$. It follows from $\bfl^*\in \widehat{N}_S(\bfy^*)$ in \eqref{KKT-point}  and \eqref{NS-exp} that
\begin{eqnarray}\label{yTlT}
\bfy^*_{\overline{J}_* }\neq 0,~ \bfl^*_{\overline{J}_* }=0,~~~~
 \bfy^*_{J_* }= 0,~\bfl^*_{J_* }\geq0. 
\end{eqnarray}
F3) If
there is a $j\in \overline{J}_*$ such that $y_j^*>0$ but $y_j\leq 0$, then this causes the following contradiction,
\begin{eqnarray*}\delta_0^2> \|\bfx^*-\bfx \|^2+\|\bfy^*-\bfy\|^2 
 \geq |y_j^*-y_j|^2\geq |y_j^*|^2\geq\delta_0^2.\end{eqnarray*}
Therefore, for any $j\in \overline{J}_*$,  $y_j>0$ if $y_j^*>0$, which indicates $\|\bfy_+\|_0\geq \|\bfy^*_+\|_0=s$. Moreover, if there is a $j\in J_*$ such that $y_j>0$, then $\|\bfy_+\|_0\geq s+1$, contradicting with $\bfy\in S$. So we have
\begin{eqnarray}\label{yTyT}
\bfy_{J_* }\leq0~~{\rm due~to}~~\bfy^*_{J_* }= 0.
\end{eqnarray}
Finally, these three facts and the convexity of $f$ can conclude  that 
\begin{eqnarray*} 
f(\bfx)  
&\geq&f( \bfx^*)+\langle \nabla f( \bfx^*),\bfx-\bfx^*\rangle \\
&\overset{(\ref{KKT-point})}{=} &
 f( \bfx^*)-\langle \bfl^*,A(\bfx-\bfx^*)\rangle\\
& \overset{(\ref{yyAxx})}{=}& f( \bfx^*)-\langle \bfl^*,\bfy-\bfy^*\rangle\\
& \overset{(\ref{yTlT})}{=}& f( \bfx^*)-\langle \bfl^*_{J_*},\bfy_{J_*}-\bfy^*_{J_*}\rangle\\
& \overset{(\ref{yTlT})}{=}& f( \bfx^*)-\langle \bfl^*_{J_*},\bfy_{J_*}\rangle\\
 &\overset{(\ref{yTlT},\ref{yTyT})}{\geq}&  f( \bfx^*),\nonumber 
\end{eqnarray*}
which presents the global optimality of $(\bfx^*,\bfy^*)$ to the problem $\min  \{ f(\bfx):  (\bfx, \bfy)\in \F \cap U((\bfx^*, \bfy^*), \delta_0)\}
$, namely, $\bfx^*$ is a local minimizer of (\ref{ell-0-1}).
\qed
\end{proof}
\subsection{$\tau $-stationary points}
Our next result is about the $\tau $-stationary point of  (\ref{ell-0-1}). We say $\bfx^*\in\R^n$  is a $\tau $-stationary point of  (\ref{ell-0-1}) for some $\tau >0$ if there is a $\bfl^*\in\R^m$ such that
  \begin{eqnarray}\label{eta-point}
  \begin{cases}
~~~~\nabla f(\bfx^*)+A^\top\bfl^*&=~~{ 0} ,\\
\P_S\left( A\bfx^*-\bfb+\tau  \bfl^*\right)&\ni~~ A\bfx^*-\bfb.
 \end{cases}
  \end{eqnarray}
Hereafter, we also say $(\bfx^*, \bfl^*)$ is a \ts\ of  (\ref{ell-0-1}) if it satisfies \eqref{eta-point}. To establish the relationship between a  $\tau $-stationary point and a local/global minimizer of    (\ref{ell-0-1}), we need the concept of the strong  convexity.  A function $f$ is  strongly convex with a constant $M_f>0$ if for any $\bfx,\bfx'\in\R^n$ it satisfies
\begin{eqnarray}\label{strong-convex}
f(\bfx) \geq   f( \bfx')+
 \langle \nabla f( \bfx'),\bfx-\bfx'\rangle
+(M_f/2)\|\bfx-\bfx'\|^2.
\end{eqnarray}
Under the strong convexity, a  $\tau $-stationary point might be a global minimizer.
\begin{theorem}[$\tau $-stationary points and local/global minimizers]\label{nec-suff-opt-con-sta} The following results hold for the problem (\ref{ell-0-1}).
\begin{itemize}
\item[a)]Suppose $A_{J_*}$ is full row rank.  A local minimizer $\bfx^*$ is also a $\tau $-stationary point 
\begin{itemize}
\item either for any~  $\tau>0$ if $\|(A\bfx^*- \bfb)_+\|_0<s$
\item  or for any $0<\tau\leq\tau_*$ if $\|(A\bfx^*- \bfb)_+\|_0=s$, where
  \begin{eqnarray}\label{tau*}\tau_*:=\frac{(A\bfx^*- \bfb)_{[s]}}{\max_{i\in\N_m}\left| ( \Pi\nabla f( \bfx^*) )_i\right|}>0~~{\rm and}~~\Pi:=(A_{J_*}A_{J_*}^\top )^{-1} A_{J_*}.\end{eqnarray}
\end{itemize}
\item[b)] Suppose $f$ is convex. A  $\tau $-stationary point with $\tau>0$ is a local minimizer if  $\|(A\bfx^*-\bfb)_+\|_0=s$ and a global minimizer if $\|(A\bfx^*-\bfb)_+\|_0<s$.
\item[c)] Assume $f$ is strongly convex with  $M_f>0$. If $\bfx^*$ is a   $\tau $-stationary point with $\tau \geq \|A\|^2/M_f$, then it is also a global minimizer.
\end{itemize}
\end{theorem}
\begin{proof}  a) Denote $\bfy^*:= A\bfx^*-\bfb$. It follows from   \cref{nec-suff-opt-con-KKT} that a local minimizer $\bfx^*$  is also a KKT point. Therefore, we have \eqref{KKT-point}. The normal cone $\widehat{N}_S$  at $\bfx^*$  in \eqref{NS-exp} indicates  $\bfl^*=0$  if $\|\bfy^*_+\|_0<s$, and 
\begin{eqnarray}\label{kkt-N}
\bfl_{\overline J_*}^*=0, ~~\bfy_{\overline J_*}^*\neq0, ~~~~ 
\bfl_{  J_*}^*\geq0, ~~\bfy_{ J_*}^*=0,
  \end{eqnarray}
if $\|\bfy^*_+\|_0=s$. To prove  the $\tau $-stationarity, we only need to show $\bfy^*\in\P_S\left( \bfy^*+\tau  \bfl^*\right)$ from (\ref{eta-point}), namely,  to show 
  \begin{eqnarray}\label{eta-point-1}
 \|\bfy^*_+\|_0\leq s 
 ,~~ \lambda_i^*\left\{
\begin{array}{ll}
=0,&i\in\supp(\bfy^*),\\
\in\left[0, y^*_{[s]}/\tau \right],&i\notin\supp(\bfy^*),
\end{array}
\right.
  \end{eqnarray}
from (\ref{uPu-equ}) in  \cref{pro-eta}. If $\|\bfy^*_+\|_0<s$, then $\bfl^*=0$, which derives \eqref{eta-point-1} for any $\tau>0$.

 Now consider the case $\|\bfy^*_+\|_0=s$. It suffices to $\tau_*>0$ in \eqref{tau*} because $(A\bfx^*- \bfb)_{[s]}=y^*_{[s]}>0$. It follows from   the  full row rankness of $A_{J_*}$ and  
$$0=\nabla f(\bfx^*)+A^\top\bfl^*\overset{\eqref{kkt-N}}{=} \nabla f(\bfx^*)+A^\top_{J_*}\bfl^*_{J_*}$$ 
that 
$\bfl^*_{J_*}=-\Pi\nabla f( \bfx^*),$
 which enables us to derive \eqref{eta-point-1} by $0<\tau\leq\tau_*$ and    $$\forall~j\in J_*,~~\tau \lambda_j^*\leq\tau_* \max_i|(\Pi\nabla f( \bfx^*))_i|\overset{\eqref{tau*}}{=} (A\bfx^*- \bfb)_{[s]} =y^*_{[s]}.$$ 
  
b)  A \ts\ satisfies $\bfy^*\in\P_S\left( \bfy^*+\tau  \bfl^*\right)$, implying \eqref{eta-point-1},
 which suffices to show $\bfl^*\in \widehat{N}_S(\bfy^*)$ by \eqref{NS-exp}. This together with the first condition in (\ref{eta-point}), $\bfy^*= A\bfx^*-\bfb$ and  $\|\bfy^*_+\|_0\leq s $ exhibits $\bfx^*$  satisfying (\ref{KKT-point}), a KKT point. Then the conclusion holds immediately by \cref{nec-suff-opt-con-KKT} b).\\
 
 c) The second inclusion $\bfy^*\in\P_S\left( \bfy^*+\tau  \bfl^*\right)$ of (\ref{eta-point})   indicates 
    \begin{eqnarray}\label{eta-point-2}
 \|\bfy^*-(\bfy^*+\tau  \bfl^*)\|^2 \leq  \|\bfy -(\bfy^*+\tau  \bfl^*)\|^2
  \end{eqnarray}
  for any $\bfy\in S$, which leads to the fact that
     \begin{eqnarray} \label{eta-point-3-0} 
2 \langle \bfl^*, \bfy - \bfy^*\rangle&\leq& \|\bfy - \bfy^* \|^2/\tau.
  \end{eqnarray} 
  In addition, For any point $(\bfx, \bfy)\in\F$, namely, $\bfy=A \bfx-{\bfb} $ and $\bfy\in S$, we have \beq\label{yyxx}
  \bfy-\bfy^*=A(\bfx-\bfx^*).\eeq
   Finally, the strongly convexity of $f$ allows us to derive that
  \begin{eqnarray*}
2f(\bfx) - 2f( \bfx^*) 
 &\geq &2\langle \nabla f( \bfx^*),\bfx-\bfx^*\rangle
 + M_f \|\bfx-\bfx^*\|^2\\
  &\overset{(\ref{eta-point})}{=} &-2\langle  A^\top \bfl^*,\bfx-\bfx^*\rangle+ M_f \|\bfx-\bfx^*\|^2\\
        &\overset{(\ref{yyxx})}{=}&-2\langle \bfl^*,  \bfy-\bfy^* \rangle+ M_f \|\bfx-\bfx^*\|^2\\
   &\overset{(\ref{eta-point-3-0})}{\geq} & - \|\bfy - \bfy^* \|^2/ \tau  + M_f\|\bfx-\bfx^*\|^2 \\
      &\overset{(\ref{yyxx})}{=}& - \|A(\bfx - \bfx^*) \|^2/\tau +M_f\|\bfx-\bfx^*\|^2 \\
 &\geq& ( M_f -\|A\|^2/\tau) \|\bfx-\bfx^*\|^2.
\end{eqnarray*}
 This shows the global optimality of $(\bfx^*,\bfy^*)$ if it is an  $\tau $-stationary point with $\tau \geq \|A\|^2/M_f$.
The whole proof is finished. \qed \end{proof}

It is worth mentioning that for the case of $\|(A\bfx^*-\bfb)_+\|_0=s$, a  $\tau $-stationary point with $\tau>0$ is a local minimizer if $f$ is convex. Moreover, it is able to be global minimizer if  $f$ is  further strong convex and $\tau\geq\|A\|^2/M_f.$ Based on  \cref{nec-suff-opt-con-KKT} and   \cref{nec-suff-opt-con-sta}, we have the following relationships among $\tau $-stationary points,  KKT points and local/global minimizers.
\begin{corollary}\label{cor-relation} The relationships among $\tau $-stationary points, KKT points and local/global minimizers
 are displayed in \cref{relation}. 
\begin{figure}[!th] 
  \centering\vspace{-3mm}
  \includegraphics[width=.8\textwidth]{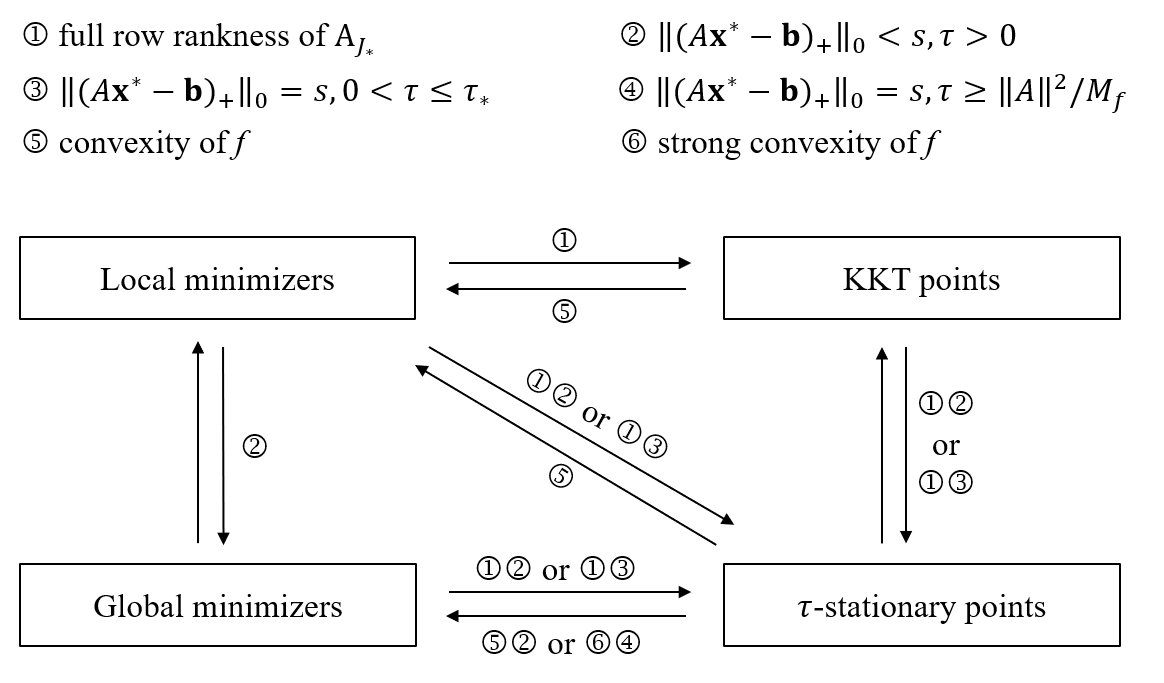}
  \caption{Relationships among four types of points.}\label{relation}
\end{figure}
\end{corollary}
We would like to point out that  the first-order sufficient and necessary optimality conditions are established for the model \eqref{ell-0-1}, where the objective function $f$ is assumed to be continuously differentiable. In fact, those results can be extended for the model with $f$ being sub-differentiable, for example, $f(\bfx)=\|\bfx\|_1+\eta\|\bfx\|^2$ in the problem \eqref{1-bit-cs-l01}.

 \section{Newton Method}\label{sec:Newton}
This section casts a Newton-type algorithm that aims at solving \eqref{ell-0-1} to find a \ts.  Before embarking that, in the sequel, we assume the following assumptions on $f$ for a given \ts\ $\bfx^*$ of \eqref{ell-0-1}. 
\begin{assumption}\label{ass-1}  The function $f$ is twice continuously differentiable on $\R^n$,  $\nabla^2 f(\bfx^*)$ is positive definite  and $A_{J_*}$ is full row rank, where $J_*$  is given by (\ref{gamma*}).  
\end{assumption}
\begin{assumption}\label{ass-2} The Hessian matrix $\nabla^2 f(\cdot)$ is locally Lipschitz continuous around $\bfx^*$ with a constant $L_*>0$, namely, 
\beq\label{Hessian-Lip}\|\nabla^2 f(\bfx)-\nabla^2 f(\bfx')\|\leq L_* \|   \bfx -\bfx' \| \eeq
for any $\bfx,\bfx'$ in the neighbourhood of $\bfx^*$.
\end{assumption}
 \subsection{$\tau$-stationary equations}  
To employ the Newton method,   we first convert a \ts\ satisfying  \eqref{eta-point} to an equation system. 
 \begin{theorem}\label{sta-eq} A point $\bfx^*$  is a \ts\ with  $\tau>0$  of (\ref{ell-0-1}) if and only if there is a $\bfl^*\in\R^m$ such that $J_*\in \T_{\tau}(A\bfx^*-\bfb+\tau  \bfl^*;s)$ and
     \begin{eqnarray}\label{sta-eq-1}
  F(\bfw^*;J_*):=\left[  
     \begin{array}{c}
    \nabla f(\bfx^*)+ A^\top_{J_*}\bfl_{J_*}^*\\
     A_{J_*}\bfx^* -\bfb_{ J_*}\\ 
     \bfl^*_{\overline J_*}\\
     \end{array}
     \right] =0.
  \end{eqnarray}
  Furthermore, for the fixed ${J}_*$, the Jacobian of $F$ takes the following form
 \begin{eqnarray}\label{sta-eq-2}
 \nabla F(\bfw^*;{J}_*)=\left[  
     \begin{array}{ccc}
    \nabla^2 f(\bfx^*)& A^\top_{J_*}&0\\
     A_{J_*}&0&0\\
     0&0&I\\ 
     \end{array}
     \right],
  \end{eqnarray} 
  which is non-singular if Assumption \ref{ass-1} holds.
\end{theorem}
\begin{proof} The second claim can be made by Assumption \ref{ass-1} obviously. We only prove the first conclusion. Let $\Gamma_s^*$ and $\Gamma_0^*$ be defined as \eqref{notation-z} in which $\bfz$ is replaced by $\bfz^*:=A\bfx^*-\bfb+\tau  \bfl^*$. 
Necessity. By \eqref{eta-point}, there is a $\bfl^*\in\R^m$ satisfying
\begin{eqnarray*}A\bfx^*-\bfb&\in&\P_S(A\bfx^*-\bfb+\tau  \bfl^*)=\P_S( \bfz^*)\\
&\overset{\eqref{psz}}{=}&\{(0;\bfz_{\overline T}^*):~T\in\T_{\tau}(\bfz^*;s)\}\\
&=& \left\{\left(0;(A\bfx^*-\bfb)_{\overline T}+\tau \bfl_{\overline T}^* \right):~T\in\T_{\tau}(\bfz^*;s)\right\}, 
\end{eqnarray*}
 which indicates that there is  a $T_*\in \T_{\tau}(\bfz^*;s)$ such that $A\bfx^*-\bfb=(0;(A\bfx^*-\bfb)_{ \overline   T_*}+\tau \bfl_{\overline  T_*}^*)$, sufficing to $A_{T_*}\bfx^* -\bfb_{T_*}=0$ and $\bfl^*_{\overline T_*}=0$ and hence $\nabla f(\bfx^*)+ A^\top_{T_*}\bfl_{T_*}^*=0$. Now we claim that $T_*=J_*$. Obviously, the definition of $J_*$ in (\ref{gamma*}) indicates $T_*\subseteq J_*$. Suppose there is a $j\in J_*$ but $j\notin T_*$. Then we get $A_{j}\bfx^* -b_{j}=0$ by (\ref{gamma*}) and $\lambda^*_j=0$  by $j \in \overline T_*$.  These show $z_j^*=A_{j}\bfx^* -b_{j}+\tau \lambda^*_j=0$ and hence $j\in \Gamma_0^*\subseteq T_*$, a contradiction. So $T_*=J_*$. \\
Sufficiency.  $\nabla f(\bfx^*)+ A^\top\bfl^*=0$ due to the first and third equations in \eqref{sta-eq-1}. The last two equations in \eqref{sta-eq-1} indicate  that 
\begin{eqnarray*} A\bfx^*-\bfb&=&(0;(A\bfx^*-\bfb)_{\overline J_*}),\\
\bfz^*=A\bfx^*-\bfb+\tau  \bfl^*&=&(\tau\bfl_{J_*}^*; (A\bfx^*-\bfb)_{\overline J_*}).\end{eqnarray*} 
These and  the definition of $J_*=(\Gamma^*_+\setminus\Gamma^*_{s})\cup \Gamma^*_0\in \T_{\tau}(\bfz^*;s)$ from \eqref{T-z} derive $$\forall i\in\Gamma_s^*\subseteq \overline J_*,~(A\bfx^*-\bfb)_{i}= z_i^*\geq  z_j^*=\tau\lambda_{j}^*\geq0,\forall j\in J_*,$$ which shows $A\bfx^*-\bfb\in\P_S(\bfz^*)$ immediately by  \cref{pro-eta}. Namely, $\bfx^*$  is a \ts\ of the problem \eqref{ell-0-1}.   \qed  
\end{proof}

\subsection{Algorithmic design}
For notational simplicity, hereafter, we always let
  \begin{eqnarray}
\label{ghwz-k}\bfw^k&:=&(\bfx^k;\bfl^k), ~~~~ \bfz^k:= A \bfx^k -\bfb + \tau\bfl^k,
  \end{eqnarray}
  if no extra explanations are given. These allow  us to borrow the notation in \eqref{notation-z} and \eqref{T-z}. In addition, we rewrite $\T(\bfz^k;s)$ in \eqref{T-z} as 
    \begin{eqnarray}\label{T-tau-T}
 \T_{\tau}(\bfz^k;s):=  \T(\bfz^k;s),
  \end{eqnarray}
  since $\bfz^k$ is associated with the parameter $\tau$. 
Recall that  \eqref{notation-z}, for  $\bfz^k$ we have  $\Gamma$-related index sets: $\Gamma_+^k, \Gamma_s^k,\Gamma_0^k$ and $\Gamma_-^k$. Then  
\beq\label{tjt-k}\forall~T_k\in \T_{\tau}(\bfz^k;s),~~~ 
 T_k &\overset{\eqref{T-z}}{=}& (\Gamma_+^k\setminus\Gamma_s^k) \cup\Gamma_0^k,~~~\overline T_k= \Gamma_s^k\cup\Gamma_-^k.
  \eeq
To differ with the superscripts used in the $\Gamma$-related index sets, we make use of the subscript in $T_k$.   Now turn our attention to solve the equations \eqref{sta-eq-1}, in which, however, $J_*\in\T_{\tau}(\bfz^*;s)$ is unknown. Therefore, to proceed the Newton method, we need to find such a $J_*$, which will be adaptively updated by using the  approximation of $\bfw^*$. More precisely,  let $\bfw^k$ be the current point, we first select a $T_k\in\T_{\tau}(\bfz^k;s)$ and then find Newton direction $\bfd^k$ by solving the following linear equations:
 \beq \label{newton-dir-0}
 \nabla F(\bfw^k;T_k)~\bfd = - F(\bfw^k;T_k). 
 \eeq
Let $\bfd^k=(\bfu^k;\bfv^k)$ with $\bfu^k\in\R^n$ and $\bfv^k\in\R^{m}$. By \eqref{sta-eq-2}, the above equation yields
 \beq \label{newton-dir} 
 \left[  
     \begin{array}{ccc}
   \nabla^2 f(\bfx^k)& A^\top_{ T_k}&0\\
    A_{ T_k}&0&0\\
    0&0&I
     \end{array}
     \right] 
     \left[  
     \begin{array}{l}
    \bfu^k\\
    \bfv^k_{ T_k}\\
    \bfv^k_{\overline T_k}
     \end{array}
     \right]
     &=&-\left[  
     \begin{array}{c}
    \nabla f(\bfx^k) + A^\top_{ T_k}\bfl_{  T_k}^k\\
    A_{  T_k}\bfx^k-\bfb_{  T_k}\\
    \bfl_{\overline T_k}^k
     \end{array}
     \right].
 \eeq
 The framework of our proposed method is summarized in  \cref{Alg-NL01}.

\begin{algorithm} 
	\caption{\nhs: Newton method for Heaviside set constrained optimization}
	\begin{algorithmic}[1] \label{Alg-NL01}
		\STATE Initialize $\bfw^0=(\bfx^0;\bfl^0)$. Give the parameter $\tau>0$, the maximum number of iteration \texttt{maxIt} and the tolerance \texttt{tol}. Select $s\in\N_m$ and set $k :=0$.
		\STATE Pick $T_0\in\mathbb{T}_{\tau}(\bfz^0;s)$.
\IF{ $k\leq\texttt{maxIt}$ and $\|F(\bfw^k;T_k)\|>\texttt{tol}$}			
		\STATE Update $\bfd^k$ by solving \eqref{newton-dir}.
		\STATE Update $\bfw^{k+1}~=~\bfw^k+\bfd^k$.
		\STATE Update $T_{k+1}\in\mathbb{T}_{\tau}(\bfz^{k+1};s)$  and set $k := k+1$.   		
\ENDIF		
\RETURN $\bfw^{k}$.
	\end{algorithmic}
\end{algorithm}

\begin{remark}\label{rem:complexity}
With regard to  \cref{Alg-NL01}, we have some observations.
\begin{itemize}
\item[i)] One of the halting conditions makes use of $\|F(\bfw^k;T_k)\|$. The reason behind this is that if we find a $T_k\in\T_\tau(\bfz^k;s)$ satisfying $\|F(\bfw^k;T_k)\|=0$, then $\bfw^k$ is a \ts\ of the problem (\ref{ell-0-1}) by  \cref{sta-eq}.
\item[ii)] The major concern  is made on the existence of $\bfd^k$, namely the existence of solutions of the linear equations  (\ref{newton-dir}). Suppose $f$ is strongly convex, such as $f= \|D\bfx\|^2$ in (\ref{HM-SVM})  with $D_{nn}>0$. Then $\Theta_k:=\nabla^2 f(\bfx^k)$ is invertible.
In order to solve (\ref{newton-dir}), one could address 
\allowdisplaybreaks \beq \label{newton-dir-1}
 \left\{\begin{array}{rll}
 A_{ T_k}\Theta_k^{-1} A^\top_{ T_k} \bfv^k_{ T_k} &=& A_{  T_k}\bfx^k-\bfb_{ T_k} - A_{  T_k}\Theta_k^{-1}\Big[ \nabla f(\bfx^k) + A^\top_{  T_k}\bfl_{  T_k}^k\Big], \\
 \bfu^k&=& -\Theta_k^{-1} \Big[  \nabla f(\bfx^k) + A^\top_{  T_k}\bfl_{ T_k}^k+A^\top_{  T_k} \bfv^k_{  T_k}\Big],\\
\bfv^k_{\overline T_k}&=& -\bfl_{\overline T_k}^k.
 \end{array}
 \right. 
 \eeq
\item[iii)] When it comes to the computational complexity, if $\Theta_k$ is a diagonal matrix with diagonal elements being non-zeros (e.g., $f$ is the one  in   (\ref{HM-SVM-l01})  or (\ref{1-bit-cs-l01})), then the complexity of computing $A_{  T_k}\Theta_k^{-1}A^\top_{ T_k} $ is $\mathcal{O}(n|  T_k|^2)$ and tackling the first linear equation in (\ref{newton-dir-1}) needs complexity at most $\mathcal{O}(|  T_k|^3)$. Overall the total complexity of solving  (\ref{newton-dir-1}) is $\mathcal{O}(n|  T_k|^2+|  T_k|^3)$. However, for general Hessians $\Theta_k$, the computation might be expensive when $n$ is enormous. For such scenarios, a potential way is to adopt the conjugate gradient method to solve (\ref{newton-dir}).  To select $T_{k}\in\mathbb{T}_{\tau}(\bfz^{k};s)$, we  pick  the indices of the $|\Gamma_+^k|-s+|\Gamma_0^k|$ smallest non-negative entries of $\bfz^{k}$.  The complexity of the section of $  T_{k}$ is about $O(m+s\log s)$. 

\end{itemize} 
\end{remark}

\subsection{Locally quadratic convergence}
Loosely speaking, Newton method enjoys the locally quadratic convergence property if the starting point is sufficiently close to a stationary point under some standard assumptions, such as Assumptions \ref{ass-1} and \ref{ass-2}. In the sequel, we will show that our proposed method \nhs\ enjoys this property under those assumptions.  Before which, we  define some notation and constants. Given a $\tau$-stationary point $ \bfw^*=(\bfx^*;\bfl^*)$ of the problem \eqref{ell-0-1}, let $J_*$ be given by \eqref{gamma*} and
\beq\label{constants-y-H}
\bfy^*:=A\bfx^*-\bfb,~~~ \bfz^*: =  \bfy^*+\tau \bfl^*, ~~~H(J):=\left[  
     \begin{array}{ccc}
    \nabla^2 f(\bfx^*)& A^\top_{J}\\
     A_{J}&0
     \end{array}
     \right],  
\eeq
where $J \subseteq\N_m$. Based on which, we  denote some constants by
\beq\label{constants}
C_*&:=&  2\max\{\|H(J_*)\|,1\},\nonumber\\
c_*&:= &  2/ \min\Big\{ \min_{J\subseteq J_*}\sigma_{\min}(H(J)),1\Big\}, \\
\tau_*&:=& \begin{cases}
 {y^*_{[s]}}/\max_{i}\lambda_i^*,&~~{\rm if}~~ \bfl^* \neq 0,\\
 +\infty,&~~{\rm if}~~ \bfl^* = 0.
\end{cases} \nonumber 
\eeq
It is worth mentioning that the term $\min_{J\subseteq J_*}\sigma_{\min}(H(J))$ can be  derived explicitly by only using $\nabla^2 f(\bfx^*)$ and $A_{J_*}$. For simplicity, we keep such an expression. 
Under Assumptions \ref{ass-1} and \ref{ass-2}, those constants are all well defined. Now we  present the main convergence results in the following theorem.
\begin{theorem}[Locally quadratic convergence]\label{the:quadratic}  Let  $\bfw^*$ be a {\ts} of (\ref{ell-0-1}) with $0<\tau<\tau_*$, Assumptions \ref{ass-1} and \ref{ass-2} hold and $ \tau_*,c_*, C_*$ be given by (\ref{constants}).  Let $\{\bfw^k\}$ be the  sequence generated by   \cref{Alg-NL01}. There always exists a $\delta_*>0$   such that, if the initial point satisfies $\bfw^0\in U(\bfw^*,\min\{ {\delta_*},1/ ({c_* L_*})\})$, then the following results hold. 
\begin{itemize}
\item[a)] The sequence $\{\bfd^k\}_{k\geq0}$ is well defined and $\lim_{k\rightarrow\infty}\bfd^k=0$. 
\item[b)] The whole sequence $\{\bfw^k\}$ converges to $\bfw^*$ quadratically, namely,
\beq
\|\bfw^{k+1}-\bfw^{*}\|  &\leq& 
 0.5 c_* L_*  \|\bfw^{k}-\bfw^{*} \|^2. \nonumber 
\eeq
\end{itemize}
\end{theorem}
 We would like to emphasize that $\delta_*$  can be replaced by one of its upper bounds that is able be derived explicitly through the $\tau$-stationary point $\bfw^*$. For simplicity of the proof, we assume the existence of $\delta_*$.
\section{Numerical Experiments}\label{sec:numerical}

In this section, we will conduct extensive numerical experiments of  {\tt NHST}, a variant of {\tt NHS}, by using MATLAB (R2019a) on a laptop of  $32$GB memory and Inter(R) Core(TM) i9-9880H 2.3Ghz CPU, against a few  solvers for addressing the SVM  and 1-bit CS problems. 
\subsection{Tuning $s$}
In the model \eqref{ell-0-1}, $s$ is a given integer, while being unknown for many applications in general. Therefore, it is necessary to design a proper scheme to tuning $s$ adaptively. In fact, $s$ plays  two important roles:
\begin{itemize}
\item[i)] For starters, $s$ is suggested to be a small integer instead of a large one. Taking the SVM as an example, the constraint $\|( A\bfx +{\bf1})_+\|_0\leq s$ in \eqref{HM-SVM-l01} allows $s$ samples to be misclassified. Therefore, a big value of $s$ means that too many samples will be classified incorrectly, which is clearly not what we expect.
\item[ii)] In  Remark \ref{rem:complexity},   the complexity  depends on $| T_k|=|\Gamma_+^k|-s+|\Gamma_0^k|$ by \eqref{z0+-Gamma+} if $   T_k \neq \emptyset$.
Therefore, $s$  impacts the computational speed. 
In addition, \cref{sta-eq} states that $F(\bfw^k;T_k)$  is non-singular if $\nabla^2 f(\bfx^k)$ is non-singular and $A_{  T_k}$ is full row rank. Therefore, the smaller $|T_k|$, the higher possibility of $A_{ T_k}$ being full row rank. However, the small $|  T_k|$ suggests $s$ to be large, which contradicts with the requirement in i). 
\end{itemize}
To balance them, we apply a tuning strategy as follows: starting with a slightly bigger $s$ and gradually reducing it to an acceptable scale. 
 More precisely, we initialize an integer $s_{0}\in\N_m$, and then for $k\geq1$, update $s_{k+1}$ by 
 	\beq\label{update-sk}
	s_{k+1} =\min\left\lbrace \lceil \rho_1 s_{k} \rceil, \lceil \rho_2 |\Gamma_+^k|\rceil\right\rbrace,
	\eeq
 where $\rho_1,\rho_2\in(0,1)$. We now interpolate this updating rule into   \cref{Alg-NL01} and obtain  \cref{Alg-NL01-s} as below. 
\begin{algorithm} 
	\caption{\nhsa: \nhs\ with tuning $s$}
	\begin{algorithmic}[1] \label{Alg-NL01-s}
		\STATE Initialize $\bfw^0=(\bfx^0;\bfl^0)$. Give the maximum number of iteration \texttt{maxIt}, the tolerance \texttt{tol} and parameters $\tau, 1>\rho_i>0,i=0,1,2,3$. Set $k :=0$. 
 \STATE Pick $T_0\in\T_\tau(\bfz^0; s_{0})$ with $s_{0}=\lceil \rho_0 |\Gamma_+^0|\rceil$. 
\IF{($\|F(\bfw^k;T_k)\|\geq\texttt{tol}$ or $ s_k \geq \lceil \rho_3 m\rceil$) and ($k\leq\texttt{maxIt}$)} 
	\STATE Update $\bfd^k$ by solving \eqref{newton-dir}.
	\STATE Update $\bfw^{k+1}= \bfw^k+\bfd^k$.
	\STATE Update $s_{k+1}$ by \eqref{update-sk} and $T_{k+1}\in\T_\tau(\bfz^{k+1}; s_{k+1})$. Set $k := k+1$.
\ENDIF	
\RETURN $\bfw^{k}$.
	\end{algorithmic}
\end{algorithm}
 
\begin{remark}   \cref{Alg-NL01-s} can be reduced to  \cref{Alg-NL01}  if we set $s_k=s$ for any $k\geq0$. However, there are at least three advantages of the tuning rule in \nhsa.
\begin{itemize}
\item[i)]  Starting with  a slightly bigger value of $s_0$ could accelerate the calculation at the beginning of \nhsa\ since $| T_k|=|\Gamma_+^k|-s_k+|\Gamma_0^k|$  can be small. The larger $s_k$, the smaller $| T_k|$ and thus the higher possibility  of $A_{  T_k}$ being full row rank, which results in a nice singularity condition of $\nabla F(\bfw^k;T_k)$. So for the first few steps, the method behaves steadily and fast.
 \item[ii)]  On the other hand,  
numerical experiments have demonstrated that when the sequence starts to converge, most samples' signs can be recovered correctly, namely $(A\bfx^k-\bfb)_i< 0$. This means $\Gamma_-^k$   dominants the whole index set $\N_m$, which leads to the small value of $|T_k|$. Hence, reducing $s_k$ would not cause expensive computational costs. 
\item[iii)]  We add  $s_k\leq \lceil \rho_3 m\rceil$ as one of the halting conditions, where  $\rho_3$ is a small rate (e.g, $0.001$). This controls at most $\rho_3$ percentage of signs that might be mis-recovered. Numerically experiments demonstrate that such a stopping criterion makes \nhsa\ generate relatively accurate classifications or recovery.
\end{itemize}
\end{remark}

\subsection{Implementation}

Parameters and starting points are initialized as follows: $\texttt{maxIt}=1000$, \texttt{tol}$=10^{-6}\sqrt{n}$ and $\rho_i=0.5,i=0,1,2$ and $\rho_3 = 0.001$.  Note that \nhsa\ will stop with  $s\leq \lceil \rho_3 m\rceil$, which means the smaller $\rho_3$ is, the longer time it will take to meet such a halting condition. Despite that, our numerical experiments demonstrate that  \nhsa\ is able to run fast for some datasets on large scales. However, the choice of $\rho_3$ is flexible and would not impact on \nhsa\ significantly.  For parameter $\tau$, if we fix it, then it is suggested to be tuned  for different problems, which can be done manually or by the so-called cross validation. Apart from this, empirical experience shows that we  can update it iteratively so as to select proper one automatically by the method. We actually tested the method with both schemes: fixing  $\tau$ and updating $\tau$. The former delivered sightly better results if it was chosen properly, but needs to be chosen differently for different problems. Therefore, we employ the second scheme to unify the selection process. That is, let $\tau_0=0.5$ and update $\tau_{k+1}=\tau_k/1.1$ if $k$ is a multiple of 10 and  $\tau_{k+1}=\tau_k$ otherwise. Starting points $(\bfx^0;\bfl^0)$  are set as $\bfx^0=0, \bfl^0={\bf1}$ for the SVM problems and $\bfx^0=  \overline\bfx$, $\bfl^0=\bf1$ for the 1-bit CS problems by \cite{dai2016noisy},  where 
\beq\label{starting-x0}
 \overline\bfx:= \frac{A_0^\top \bfc}{\| A_0^\top \bfc\|},~~~~
 A_0:= [\bfa_1~\bfa_2~\cdots~\bfa_m]^\top.
 \eeq
\subsection{Simulations for SVM}
The model \eqref{HM-SVM-l01} has the objective function $f(\bfx)= \|D\bfx\|^2$ with $D_{ii}=1,i\in\N_{n-1}$ and $D_{nn}\geq0$.  Let  $\bfx^*$ be a $\tau$-stationary point with $(A\bfx^* -\bfb)_{J_*}=0$, where $J_*$ is given by \eqref{gamma*}. In SVM, this means $\langle\bfa_i,\bfx^*\rangle = \pm 1$ if the sample $i\in J_*$, namely, samples fall into the hyperplanes $\langle\bfa_i,\bfx^*\rangle = \pm 1$. Those samples (belonging to the so-called support vectors) usually take a very small portion of the  total samples, namely, $|J_*|\ll m$.  Therefore,   the full row rankness of $A_{J_*}$ turns out to be a mild assumption.  To guarantee the non-singularity of $\nabla F(\bfw^*;{J_*})$ by   \cref{sta-eq}, we need Assumption \ref{ass-1}, where the non-singularity $\nabla^2 f(\bfx^*)$ can be ensured by setting $D_{nn}>0$ (e.g. $10^{-4}$ in our numerical experiments), and the full row rankness of $A_{J_*}$, a mild assumption as  mentioned above.

\subsubsection{Testing examples} 
 We select 27 real datasets with more number of samples and less number of features (i.e.,  $m\leq n$)  from three popular libraries: libsvm\footnote{\url{https://www.csie.ntu.edu.tw/~cjlin/libsvmtools/datasets/}}, uci\footnote{\url{http://archive.ics.uci.edu/ml/datasets.php}} and kaggle\footnote{\url{https://www.kaggle.com/datasets}}. 


\begin{example}[Real data in higher dimensions]\label{ex:real-data} All datasets are feature-wisely scaled to $[-1,1]$ and all the classes unequal to $1$ are treated as $-1$.  Their details are presented in \cref{Table-svm-less-m-more-n}.   The number of samples in training and testing data is denoted by $m$ and $m_t$. The sparse status  of a dataset is also given, for instance, {\tt newb} is on a large scale but sparse.
\end{example}

\begin{table}[!th]
	\renewcommand{\arraystretch}{0.75}\addtolength{\tabcolsep}{2pt}
	\caption{Data sets with less training samples and more features, namely $m\leq n$.}\vspace{-3mm}
	\label{Table-svm-less-m-more-n}
	\begin{center}
		\begin{tabular}{lllrrrr}
			\hline
		&		&		&		&	Train	&	Test	&	Sparse	\\
	Data&Descriptions&	Source	&	$n$	&	$m$	&	$m_t$	&		\\\hline
\texttt{colc}&	colon-cancer	&	libsvm	&	2000 	&	62 	&	0	&	No	\\
\texttt{dbw1}&	\multirow{4}{3.5cm}{Dbworld e-mails}	&	uci	&	4702 	&	64 	&	0	&	Yes	\\
\texttt{dbw2}&		&	uci	&	3721 	&	64 	&	0	&	Yes	\\
\texttt{dbw3}&		&	uci	&	242 	&	64 	&	0	&	Yes	\\
\texttt{dbw4}&		&	uci	&	229 	&	64 	&	0	&	Yes	\\
\texttt{fabc}&	Farm ads binary classification	&	kaggle	&	54877 	&	4143 	&	0	&	Yes	\\
\texttt{lsvt}&	Lsvt voice rehabilitation	&	uci	&	310 	&	126 	&	0	&	No	\\
\texttt{newb}&	News20.binary	&	libsvm	&	1355191 	&	19996 	&	0	&	Yes	\\
\texttt{scad}&	Scadi	&	uci	&	205 	&	70 	&	0	&	Yes	\\
\texttt{set1}&	\multirow{11}{3.5cm}{Data for software engineering teamwork assessment in education setting}	&	uci	&	84 	&	64 	&	0	&	No	\\
\texttt{set2}&		&	uci	&	84 	&	74 	&	0	&	No	\\
\texttt{set3}&		&	uci	&	84 	&	74 	&	0	&	No	\\
\texttt{set4}&		&	uci	&	84 	&	63 	&	0	&	No	\\
\texttt{set5}&		&	uci	&	84 	&	74 	&	0	&	No	\\
\texttt{set6}&		&	uci	&	84 	&	74 	&	0	&	No	\\
\texttt{set7}&		&	uci	&	84 	&	74 	&	0	&	No	\\
\texttt{set8}&		&	uci	&	84 	&	74 	&	0	&	No	\\
\texttt{set9}&		&	uci	&	84 	&	74 	&	0	&	No	\\
\texttt{set10}&		&	uci	&	84 	&	74 	&	0	&	No	\\
\texttt{set11}&		&	uci	&	84 	&	74 	&	0	&	No	\\\hline
\texttt{arce}&	Arcene	&	uci	&	10000 	&	100 	&	100	&	No	\\
\texttt{dext}&	Dexter	&	uci	&	19999 	&	300 	&	300	&	Yes	\\
\texttt{dmea}&	Detect malacious executable 	&	uci	&	531 	&	373 	&	1	&	Yes	\\
\texttt{doro}&	Dorothea	&	uci	&	100000 	&	800 	&	350	&	Yes	\\
\texttt{dubc}&	Duke breast-cancer	&	libsvm	&	7129 	&	38 	&	4	&	No	\\
\texttt{leuk}&	Leukemia	&	libsvm	&	7129 	&	38 	&	34	&	No	\\
\texttt{rcvb}&	Rcv1.binary	&	libsvm	&	47236 	&	20242 	&	20000	&	Yes	\\

			\hline
		\end{tabular}
	\end{center}
\end{table}

 To compare the performance of all methods selected in the sequel, let $\bfx$ be the solution/classifier generated by one method  and  $
 A_0$  given in \eqref{starting-x0}. We report the {\tt CPU time} and the classification accuracy defined by
\beq\label{acc}{\tt Acc}:=\left[1-  \frac{\|{\rm sgn}( A_0\bfx)-\bfc\|_0}{m}\right]\times 100\%.\eeq
We denote  {\tt Acc} the training accuracy  if   $( A_0,\bfc)$ is the training data and {\tt TAcc} the testing accuracy if $( A_0,\bfc)$ is the testing data, where $m$ is replaced by $m_t$ in \eqref{acc}.
\subsubsection{Benchmark methods} There is a vast body of work on developing methods to tackle the SVM problems.  We select a Matlab built-in solver {\tt fitclinear} and four leading ones from the machine learning community. These methods aim at solving the regularization model \eqref{SM-SVM} with different soft-margin loss functions $\ell$.   They are: {\tt HSVM} 
from the library libsvm \cite{chang2011libsvm}, where $\ell$ is the hinge loss; 
{\tt SSVM} \cite{SV1999} implemented by liblssvm 
  \cite{pelckmans2002matlab}, where $\ell$ is the squared hinge loss; 
 {\tt RSVM} \cite{wu2007robust}, where $\ell$ is the ramp loss; 
 {\tt LSVM}  from the library liblinear 
\cite{fan2008liblinear}, where $\ell$ is $\ell_2$-regularized $\ell_2$-loss. We set the parameter ${\tt - s~2}$ for {\tt LSVM} so that primal model is solved; 
  {\tt FSVM}, an abbreviation for the solver  {\tt fitclinear},
 where $\ell$ is the hinge loss. 
The first three methods are kennel-based and we choose the liner kennel for all of them.  Other involved parameters of these five methods are set as defaults.  

\subsubsection{Comparisons for  \cref{ex:real-data}} Results of six methods are reported in   \cref{table:ex2-less-m-bigger-n}, where $``--"$ denotes the results are omitted if a solver consumes too much time or requires memory that is out of the capacity of our desktop. In general, {\tt NHST} provides the best classification accuracy for all training datasets since its {\tt Acc} is the highest, e.g., {\tt Acc}$=100\%$ for datasets {\tt set1} - {\tt set11}. As for the computational speed, despite that {\tt NHST} does not run the fastest,  as a second-order method, it is considerably competitive with the other solvers. It is worth mentioning that {\tt LSVM}  and  {\tt FSVM}  are naturally expected to run very fast since they are programmed by C language.

\begin{sidewaystable} 
\centering
  \caption{Results of six solvers for \cref{ex:real-data}.}
  \label{table:ex2-less-m-bigger-n} 		
\renewcommand{\arraystretch}{0.85}\addtolength{\tabcolsep}{-3.5pt}	
		\begin{tabular}{l|cccccc|ccccccc|ccccccc}\hline
data &\multicolumn{6}{ c|}{ {\tt Acc}$\%$}&&
			\multicolumn{6}{ c| }{ {\tt Time (seconds)}}&&
			\multicolumn{6}{c}{ {\tt TAcc}$\%$}
			\\\cline{2-7} \cline{9-14} \cline{16-21}
&	{\tt FSVM}	&	{\tt HSVM}	&	{\tt LSVM}	&	{\tt RSVM}	&	{\tt SSVM}	&	{\tt NHST}	&	
&	{\tt FSVM}	&	{\tt HSVM}	&	{\tt LSVM}	&	{\tt RSVM}	&	{\tt SSVM}	&	{\tt NHST}	&	
&	{\tt FSVM}	&	{\tt HSVM}	&	{\tt LSVM}	&	{\tt RSVM}	&	{\tt SSVM}	&	{\tt NHST}	\\ \cline{2-7} \cline{9-14} \cline{16-21}
{\tt	 colc 	}&	95.16	&	100.0 	&	100.0 	&	100.0 	&	100.0 	&	100.0 	&	&	0.207 	&	0.018 	&	0.015 	&	0.493 	&	0.864 	&	0.027 	&	&	$--$	&	$--$	&	$--$	&	$--$	&	$--$	&	$--$	\\
{\tt	 dbw1 	}&	98.44	&	98.44	&	98.44	&	98.44	&	98.44	&	98.44 	&	&	0.043 	&	0.011 	&	0.002 	&	0.061 	&	7.710 	&	0.024 	&	&	$--$	&	$--$	&	$--$	&	$--$	&	$--$	&	$--$	\\
{\tt	 dbw2 	}&	98.44	&	98.44	&	98.44	&	98.44	&	98.44	&	98.44 	&	&	0.038 	&	0.010 	&	0.002 	&	0.069 	&	4.790 	&	0.012 	&	&	$--$	&	$--$	&	$--$	&	$--$	&	$--$	&	$--$	\\
{\tt	 dbw3 	}&	98.44	&	98.44	&	100.0 	&	95.31	&	100.0 	&	100.0 	&	&	0.022 	&	0.001 	&	0.000 	&	0.035 	&	0.115 	&	0.003 	&	&	$--$	&	$--$	&	$--$	&	$--$	&	$--$	&	$--$	\\
{\tt	 dbw4 	}&	98.44	&	98.44	&	100.0 	&	93.75	&	100.0 	&	100.0 	&	&	0.094 	&	0.001 	&	0.001 	&	0.027 	&	0.124 	&	0.006 	&	&	$--$	&	$--$	&	$--$	&	$--$	&	$--$	&	$--$	\\
{\tt	 fabc 	}&	99.61	&	99.86	&	99.88	&	99.44	&	$--$	&	99.88 	&	&	0.040 	&	8.430 	&	0.194 	&	96.10 	&	$--$	&	3.782 	&	&	$--$	&	$--$	&	$--$	&	$--$	&	$--$	&	$--$	\\
{\tt	 lsvt 	}&	95.24	&	98.41	&	100.0 	&	87.3	&	100.0 	&	100.0 	&	&	0.030 	&	0.007 	&	0.005 	&	0.045 	&	0.141 	&	0.019 	&	&	$--$	&	$--$	&	$--$	&	$--$	&	$--$	&	$--$	\\
{\tt	 newb 	}&	99.54	&	$--$	&	99.87	&	$--$	&	$--$	&	99.88 	&	&	0.481 	&	$--$	&	2.030 	&	$--$	&	$--$	&	5.240 	&	&	$--$	&	$--$	&	$--$	&	$--$	&	$--$	&	$--$	\\
{\tt	 scad 	}&	98.57	&	100.0 	&	100.0 	&	97.14	&	100.0 	&	100.0 	&	&	0.011 	&	0.001 	&	0.000 	&	0.015 	&	0.097 	&	0.003 	&	&	$--$	&	$--$	&	$--$	&	$--$	&	$--$	&	$--$	\\
{\tt	 set1 	}&	81.25	&	81.25	&	87.5	&	68.75	&	92.19	&	100.0 	&	&	0.017 	&	0.001 	&	0.001 	&	0.022 	&	0.039 	&	0.008 	&	&	$--$	&	$--$	&	$--$	&	$--$	&	$--$	&	$--$	\\
{\tt	 set2 	}&	87.84	&	87.84	&	90.54	&	71.62	&	93.24	&	100.0 	&	&	0.010 	&	0.001 	&	0.001 	&	0.011 	&	0.041 	&	0.001 	&	&	$--$	&	$--$	&	$--$	&	$--$	&	$--$	&	$--$	\\
{\tt	 set3 	}&	85.14	&	85.14	&	90.54	&	68.92	&	100.0 	&	100.0 	&	&	0.012 	&	0.001 	&	0.001 	&	0.013 	&	0.039 	&	0.002 	&	&	$--$	&	$--$	&	$--$	&	$--$	&	$--$	&	$--$	\\
{\tt	 set4 	}&	80.95	&	80.95	&	90.48	&	71.43	&	98.41	&	100.0 	&	&	0.011 	&	0.001 	&	0.000 	&	0.019 	&	0.038 	&	0.002 	&	&	$--$	&	$--$	&	$--$	&	$--$	&	$--$	&	$--$	\\
{\tt	 set5 	}&	81.08	&	81.08	&	81.08	&	33.78	&	89.19	&	100.0 	&	&	0.013 	&	0.001 	&	0.001 	&	0.012 	&	0.039 	&	0.001 	&	&	$--$	&	$--$	&	$--$	&	$--$	&	$--$	&	$--$	\\
{\tt	 set6 	}&	86.49	&	86.49	&	93.24	&	70.27	&	97.3	&	100.0 	&	&	0.011 	&	0.001 	&	0.001 	&	0.011 	&	0.038 	&	0.001 	&	&	$--$	&	$--$	&	$--$	&	$--$	&	$--$	&	$--$	\\
{\tt	 set7 	}&	90.54	&	90.54	&	91.89	&	70.27	&	100.0 	&	100.0 	&	&	0.013 	&	0.001 	&	0.000 	&	0.013 	&	0.038 	&	0.001 	&	&	$--$	&	$--$	&	$--$	&	$--$	&	$--$	&	$--$	\\
{\tt	 set8 	}&	83.78	&	83.78	&	90.54	&	67.57	&	98.65	&	100.0 	&	&	0.010 	&	0.001 	&	0.001 	&	0.017 	&	0.038 	&	0.001 	&	&	$--$	&	$--$	&	$--$	&	$--$	&	$--$	&	$--$	\\
{\tt	 set9 	}&	83.78	&	83.78	&	93.24	&	67.57	&	98.65	&	100.0 	&	&	0.013 	&	0.001 	&	0.001 	&	0.014 	&	0.038 	&	0.001 	&	&	$--$	&	$--$	&	$--$	&	$--$	&	$--$	&	$--$	\\
{\tt	 set10 	}&	79.73	&	79.73	&	89.19	&	72.97	&	98.65	&	100.0 	&	&	0.011 	&	0.001 	&	0.001 	&	0.011 	&	0.038 	&	0.002 	&	&	$--$	&	$--$	&	$--$	&	$--$	&	$--$	&	$--$	\\
{\tt	 set11 	}&	83.78	&	83.78	&	90.54	&	66.22	&	95.95	&	100.0 	&	&	0.011 	&	0.001 	&	0.001 	&	0.029 	&	0.038 	&	0.001 	&	&	$--$	&	$--$	&	$--$	&	$--$	&	$--$	&	$--$	\\
{\tt	 arce 	}&	100.0 	&	100.0 	&	100.0 	&	100.0 	&	100.0 	&	100.0 	&	&	0.040 	&	0.965 	&	0.077 	&	0.017 	&	8.390 	&	0.044 	&	&	81.00 	&	86.00 	&	87.00 	&	86.00 	&	83.00 	&	86.00 	\\
{\tt	 dext 	}&	100.0 	&	100.0 	&	100.0 	&	100.0 	&	100.0 	&	100.0 	&	&	0.008 	&	0.167 	&	0.009 	&	0.055 	&	81.60 	&	0.023 	&	&	92.33 	&	92.33 	&	92.67 	&	92.33 	&	86.33 	&	92.33 	\\
{\tt	 dmea 	}&	100.0 	&	100.0 	&	100.0 	&	98.39	&	100.0 	&	100.0 	&	&	0.008 	&	0.010 	&	0.003 	&	0.518 	&	1.100 	&	0.030 	&	&	100.0 	&	100.0 	&	100.0 	&	100.0 	&	100.0 	&	100.0 	\\
{\tt	 doro 	}&	100.0 	&	100.0 	&	100.0 	&	100.0 	&	$--$	&	100.0 	&	&	0.026 	&	8.450 	&	0.078 	&	0.564 	&	$--$	&	0.095 	&	&	93.14 	&	93.14 	&	92.86 	&	93.14 	&	$--$	&	93.14 	\\
{\tt	 dubc 	}&	100.0 	&	100.0 	&	100.0 	&	100.0 	&	100.0 	&	100.0 	&	&	0.010 	&	0.035 	&	0.019 	&	0.007 	&	5.970 	&	0.010 	&	&	100.0 	&	75.00 	&	100.0 	&	75.00 	&	75.00 	&	100.0 	\\
{\tt	 leuk 	}&	100.0 	&	100.0 	&	100.0 	&	100.0 	&	100.0 	&	100.0 	&	&	0.010 	&	0.044 	&	0.022 	&	0.006 	&	5.990 	&	0.010 	&	&	76.47 	&	82.35 	&	73.53 	&	82.35 	&	85.29 	&	85.29 	\\
{\tt	 rcvb 	}&	98.97	&	98.96	&	99.72	&	$--$	&	$--$	&	99.86 	&	&	0.092 	&	180.0 	&	0.256 	&	$--$	&	$--$	&	0.770 	&	&	96.37 	&	96.37 	&	96.38 	&	$--$	&	$--$	&	95.49 	\\
\hline 
\end{tabular}
\end{sidewaystable}

\subsection{Simulations for 1-bit CS}

The model \eqref{1-bit-cs-l01} has the objective function $\phi(\bfx)+\eta\|\bfx\|^2$, where $\eta\geq0$. Note that some $\phi(\bfx)$ are not twice continuously differentiable, such as the $\ell_1$ norm or the log penalty. Hence, we will solve the problem with the smoothing $\ell_q$ norm, i.e., $\phi(\bfx)=\sum_{i=1}^n (x_i^2+\varepsilon)^{q/2}$ with fixing $q=0.9$ and $\varepsilon=1/n$ for simplicity.
Other associated parameters in \eqref{1-bit-cs-l01}  are set as $\epsilon=0.001$  and $\eta=0.07$.  

\subsubsection{Testing examples}
\begin{example}[Independent covariance \cite{yan2012robust, dai2016noisy}]\label{ex:cs-ind} Entries of $A_0$ and the nonzero entries of the ground truth $k_*$-sparse vector $\bfx^*\in\R^n$ (i.e., $\|\bfx^*\|_0\leq k_*$)  are generated from the independent and identically distributed (i.i.d.) samples of the standard Gaussian distribution $\mathcal{N}(0,1)$. Then  $\bfx^*$ is normalized to be a unit vector. Let $\bfc^*={\rm sgn}(A_0\bfx^*)$ and $\tilde\bfc ={\rm sgn}(A_0\bfx^*+ \xi)$, where entries of the noise $\xi$  are the i.i.d. samples of $\mathcal{N}(0,0.1^2)$. Finally,  we randomly select $\lceil r m\rceil$ entries in $\tilde\bfc $ and flip their signs, and the flipped vector is denoted by  $\bfc$, where  $r$  is the flipping ratio.
\end{example}

\begin{example}[Correlated covariance \cite{huang2018robust}]\label{ex:cs-cor} Rows of $ A_0$ are generated from i.i.d. samples  of $\mathcal{N}(0,\Sigma)$ with $\Sigma_{ij}=2^{-|i-j|}, i,j\in\N_n$. Then $\bfx^*, \bfc^*$ and $\bfc$ are generated the same as in \cref{ex:cs-ind}.
\end{example}
To compare the performance of all methods selected in the sequel, let $\bfx$ be the solution generated by one method. We report  the  CPU {\tt time}, the signal-to-noise ratio ({\tt SNR}$:=  -20{\log}_{10}\| \bfx -\bfx^*\|$), the hamming error ({\tt HE}$:=   \|{\rm sgn}( A_0\bfx)-\bfc^*\|_0/m$) and the hamming distance ({\tt HD}$:= \|{\rm sgn}( A_0\bfx)-\bfc\|_0/m$), defined  as follows, where the larger {\tt SNR} (or the smaller  {\tt HE} or {\tt HD}) means the better recovery.

\subsubsection{Benchmark methods}  Four state-of-the-art solvers are selected for comparisons with our method {\tt NHST}. They are {\tt PDASC}
   \cite{huang2018robust}, {\tt BIHT} 
  \cite{jacques2013robust}, {\tt AOP} 
   \cite{yan2012robust}  and {\tt PBAOP} 
  \cite{huang2018pinball}. The last three ones are required to specify the true sparsity level $k_*$. In addition, the last two solvers also need the flipping ratio, denoted by $\tilde r$. As stated by \cite{yan2012robust}, there are three options. To make the comparisons more fairly, we choose the first one,  namely, setting $\tilde r=\|{\rm sgn}( A_0\bfx)-c\|_0/m$, where $\bfx$ is the solution generated by {\tt BIHT}. The rest of parameters for each method  are chosen as defaults. Finally, all methods start with the initial point $\bfx^0= \overline \bfx$  given in \eqref{starting-x0}, and their final solutions are normalized to have a unit length.

\subsubsection{Comparisons for  \cref{ex:cs-ind} and \cref{ex:cs-cor}} 
We now apply the five methods into solving  two examples under different scenarios. For each scenario, we report average results over $500$ instances if $n\leq1000$ and $20$ instances otherwise, which are summarized as below.

\begin{itemize}
\item[i)] \textit{Effect to $k_*$} from $\{2,3,\cdots,10\}$ with fixing $n=256, m=64$ and  $r=0.05$. As shown in \cref{fig:ex2-cs-ind-s},  it can be clearly seen that {\tt NHST} gets the smallest {\tt HD} and {\tt HE} under each $k_*\leq9$ for both examples. This is because the solution obtained by {\tt NHST} satisfying the constraint in \eqref{ell-0-1}, which means only a tiny portion of samples allowing for having wrong signs. Consequently, {\tt HD} is expected to be relatively small. Moreover, {\tt NHST} delivers the highest {\tt SNR} for all cases. 
   \begin{figure}[H] 
\centering 
\begin{subfigure}{.325\textwidth}
	\centering
	\includegraphics[width=.97\linewidth]{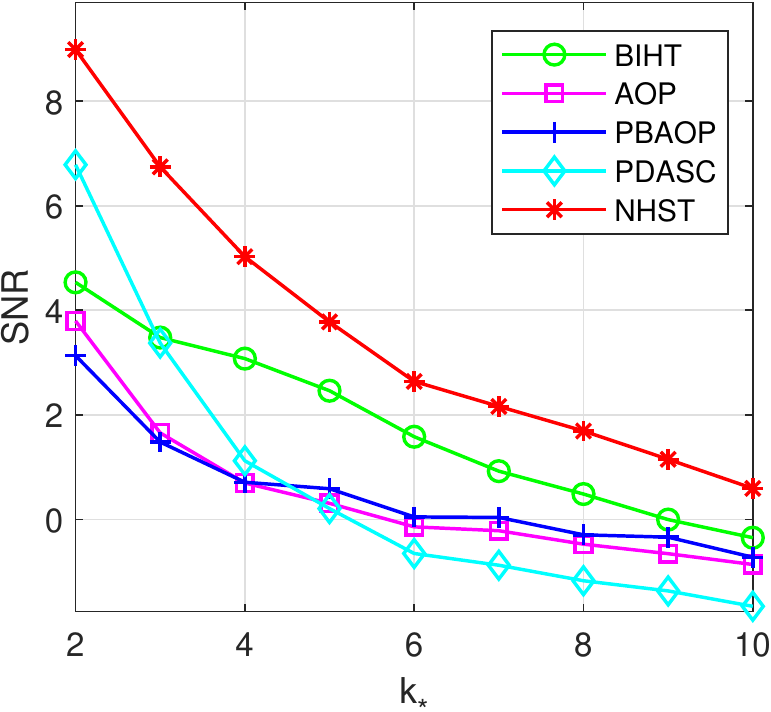}
\end{subfigure}
\begin{subfigure}{.325\textwidth}
	\centering
	\includegraphics[width=1\linewidth]{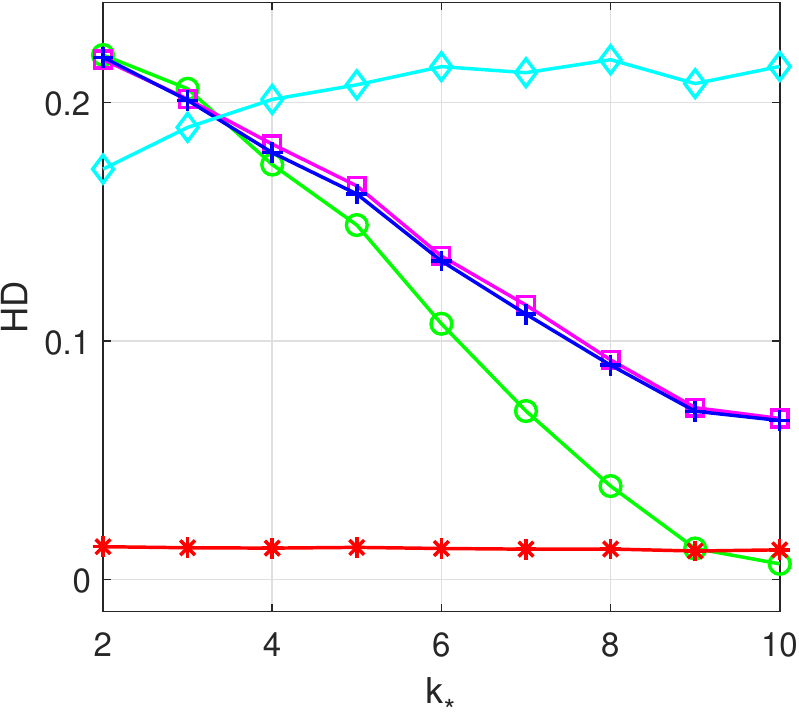}
\end{subfigure}
\begin{subfigure}{.325\textwidth}
	\centering
	\includegraphics[width=1\linewidth]{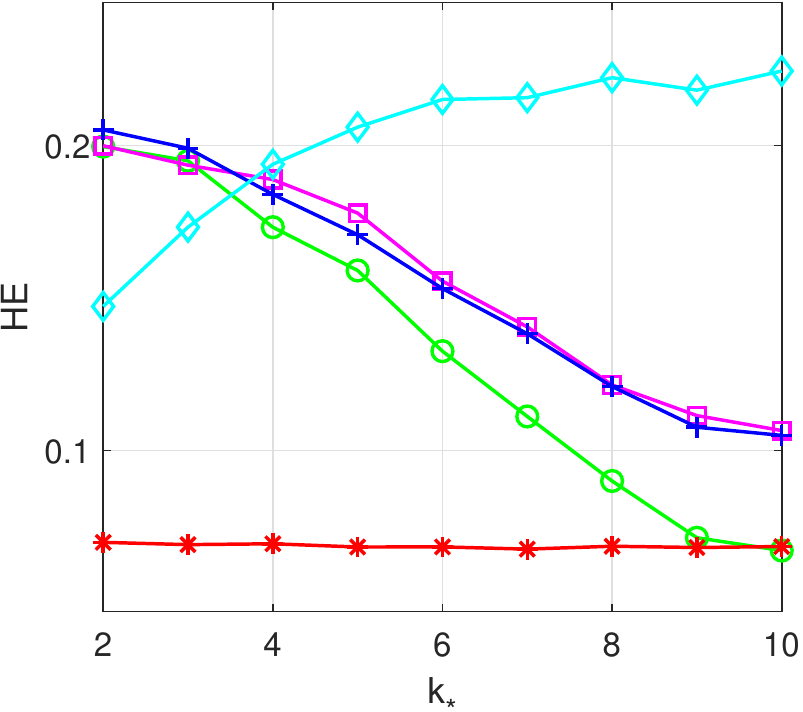}
\end{subfigure}
\caption{Effect to $k_*$ \Cref{ex:cs-cor}.}
\label{fig:ex2-cs-ind-s} 
\end{figure}
\item[ii)] \textit{Effect to $m/n$} from $\{0.1,0.2,\cdots,0.7\}$ with fixing $n=256, k_*=\lceil0.01n\rceil$ and  $r=0.05$. As shown in \cref{fig:ex2-cs-ind-r}, again, {\tt NHST} gets the smallest {\tt HD} and {\tt HE} for each $k_*$.  In terms of {\tt SNR}, {\tt NHST} outperforms others when $m/n<0.4$ and {\tt PDASC} behaves outstandingly for the bigger $m/n$. Similar observations can be seen for \cref{ex:cs-ind} and were omitted.

 \begin{figure}[H]  
\centering  
\begin{subfigure}{.32\textwidth}
	\centering
	\includegraphics[width=1\linewidth]{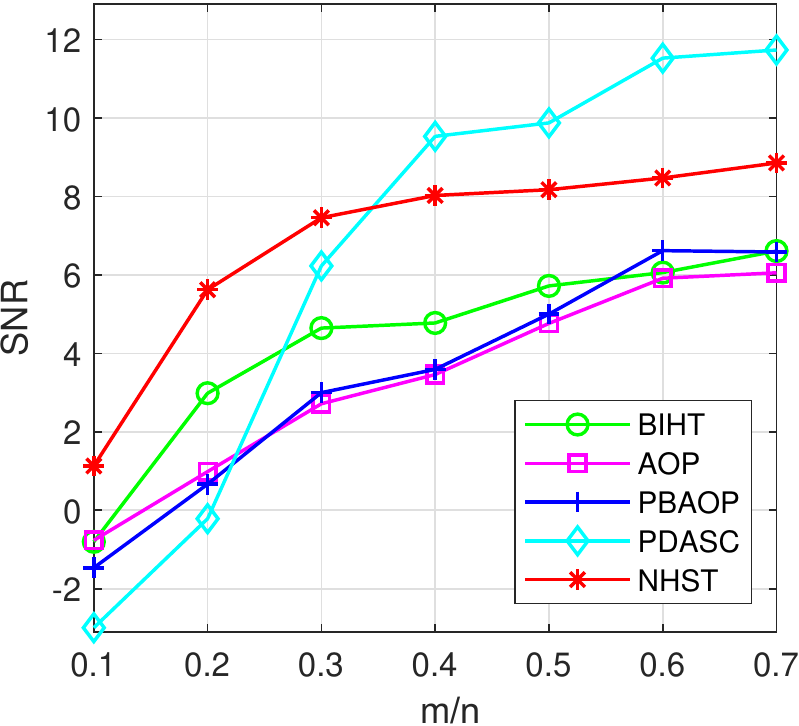}
\end{subfigure}
\begin{subfigure}{.325\textwidth}
	\centering
	\includegraphics[width=1\linewidth]{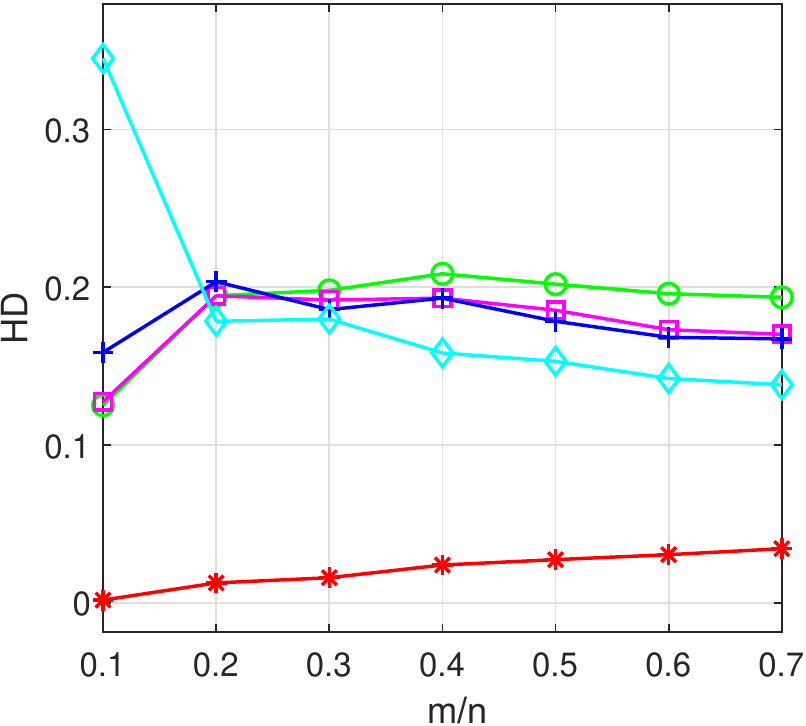}
\end{subfigure}
\begin{subfigure}{.325\textwidth}
	\centering
	\includegraphics[width=1\linewidth]{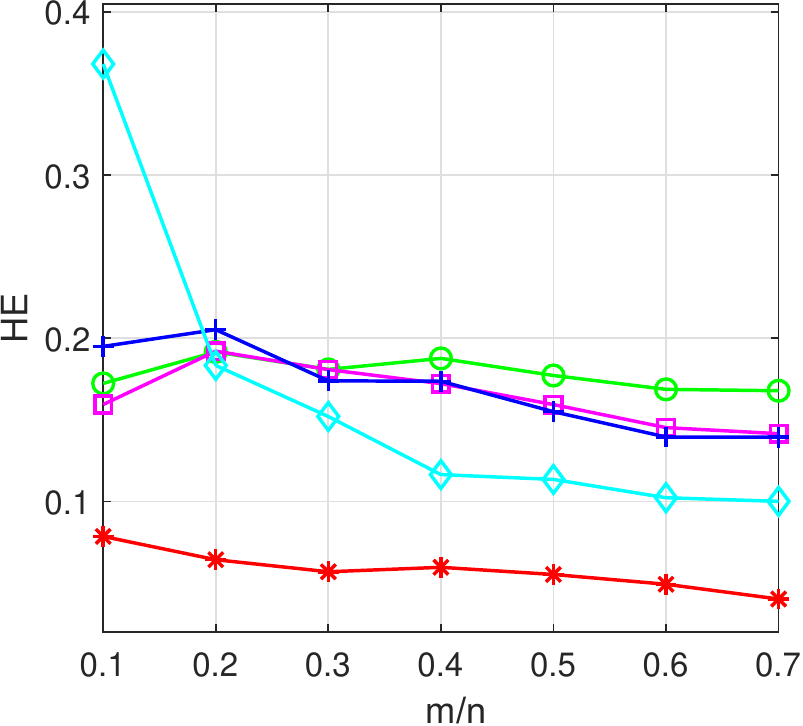}
\end{subfigure} 
\caption{Effect to $m/n$  for  \cref{ex:cs-cor}.}
\label{fig:ex2-cs-ind-r} 
\end{figure}

\item[iii)] \textit{Effect to $r$} from $\{0.02,0.02,\cdots,0.2\}$ with fixing $n=256, m=64$ and  $k_*=\lceil0.01n\rceil$.   Average results for \cref{ex:cs-cor} are reported  in \cref{fig:ex2-cs-ind-r}. As expected, {\tt NHST} behaves the best in terms of delivering the highest {\tt SNR}, the lowest {\tt HD} and {\tt HE} for each scenario. Again, we omitted the similar results for \cref{ex:cs-ind}.

 \begin{figure}[!th]  
\centering
\begin{subfigure}{.325\textwidth}
	\centering
	\includegraphics[width=.97\linewidth]{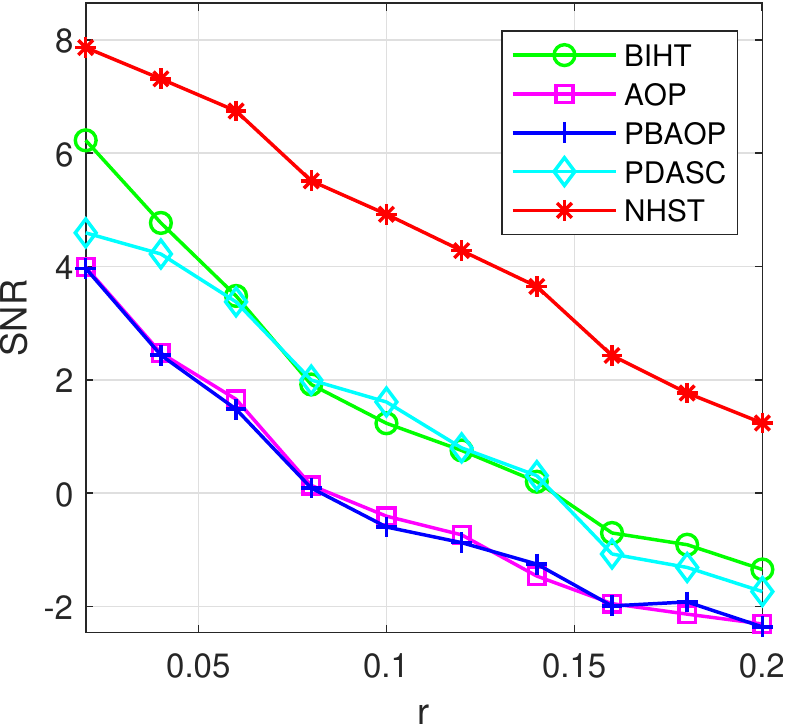}
\end{subfigure}
\begin{subfigure}{.325\textwidth}
	\centering
	\includegraphics[width=1\linewidth]{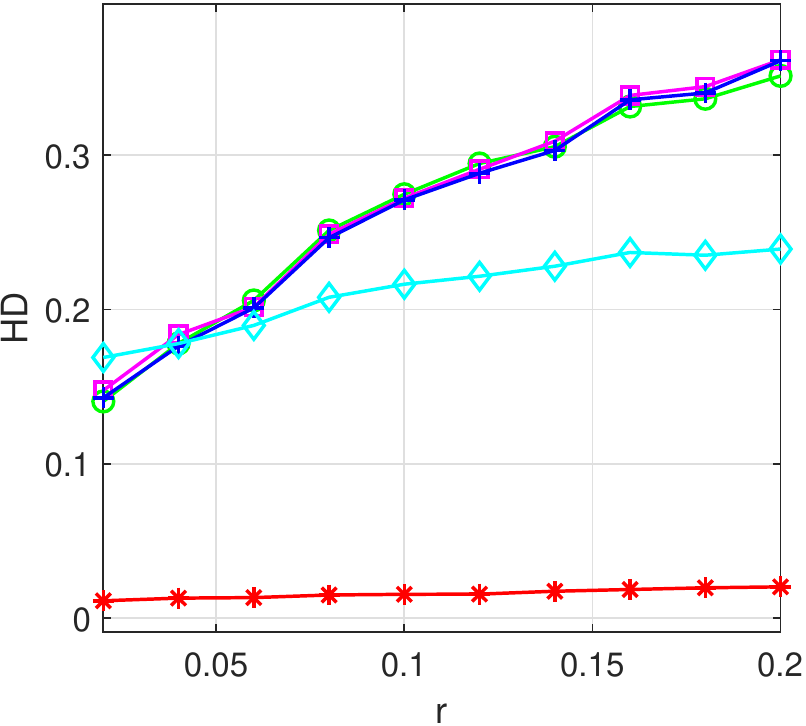}
\end{subfigure}
\begin{subfigure}{.325\textwidth}
	\centering
	\includegraphics[width=1\linewidth]{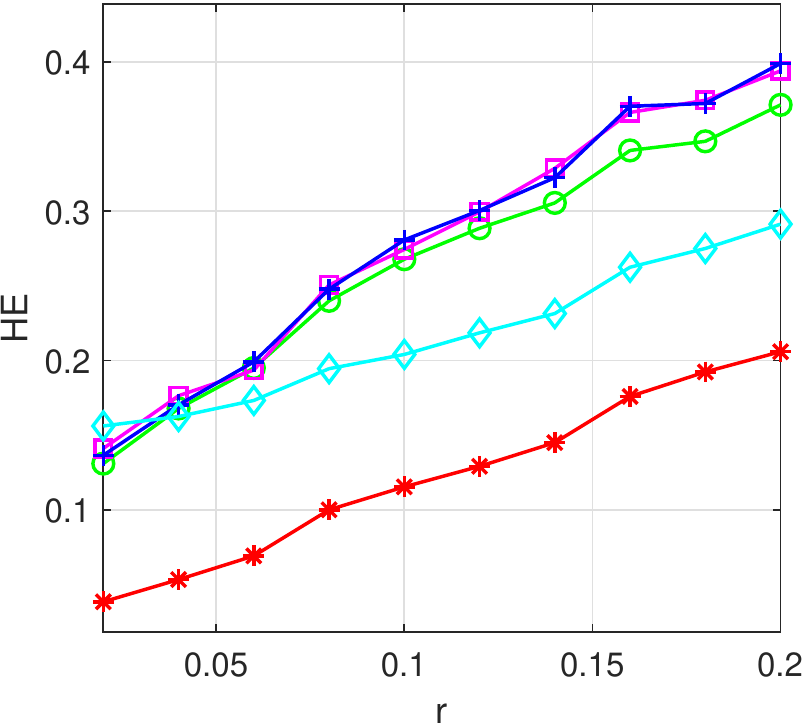}
\end{subfigure} 
\caption{Effect to $r$ for  \cref{ex:cs-cor}.}
\label{fig:ex2-cs-ind-r} 
\end{figure}
\item[iv)] \textit{Effect to $n$} from $\{1,2,3,4\}\times 5000$ with fixing $m=\lceil0.25n\rceil, k_*=\lceil0.01n\rceil$ and  $r=0.05$. We record the average results in  \cref{fig:ex2-cs-ind-n}. Obviously,   {\tt NHST} achieves the highest {\tt SNR}, lowest {\tt HD} and {\tt HE} but with consuming the shortest computational time. So it performs the best.
\end{itemize}

\begin{table}[!th]
	\renewcommand{\arraystretch}{0.85}\addtolength{\tabcolsep}{0.2pt}
	\caption{Effect to the bigger values of $n$.}\vspace{-4mm}
	\label{fig:ex2-cs-ind-n}
	\begin{center}
		\begin{tabular}{lccccccccccc}
			\hline
			&\multicolumn{5}{c}{ \cref{ex:cs-ind}}&&\multicolumn{5}{c}{ \cref{ex:cs-cor}}\\\cline{2-6}\cline{8-12}
$n$  	&	    {\tt BIHT} 	&	  {\tt AOP} 	&	   {\tt PBAOP}     	&	  {\tt PDASC} 	&	  {\tt NHST} 	&&	    {\tt BIHT} 	&	  {\tt AOP} 	&	   {\tt PBAOP}     	&	  {\tt PDASC} 	&	  {\tt NHST} 	\\\hline
&\multicolumn{5}{c}{{\tt SNR}}&&\multicolumn{5}{c}{{\tt SNR}}\\\cline{2-6}\cline{8-12}
5000	&	4.485 	&	1.883 	&	2.360 	&	0.787 	&	5.753 	&	&	4.985 	&	1.762 	&	2.193 	&	1.417 	&	5.420\\ 
10000	&	4.399 	&	1.830 	&	2.081 	&	-0.133 	&	5.478 	&	&	4.826 	&	1.574 	&	1.939 	&	0.004 	&	5.470 \\
15000	&	4.357 	&	1.861 	&	2.186 	&	-1.060 	&	5.455 	&	&	4.889 	&	1.692 	&	2.024 	&	-0.839 	&	5.402 \\
20000	&	4.425 	&	1.770 	&	2.242 	&	-1.011 	&	5.472 	&	&	5.008 	&	1.895 	&	2.203 	&	-1.456 	&	5.143 \\
\hline

&\multicolumn{5}{c}{{\tt HD}}&&\multicolumn{5}{c}{{\tt HD}}\\\cline{2-6}\cline{8-12}

5000	&	0.167 	&	0.161 	&	0.162 	&	0.321 	&	0.034 	&	&	0.168 	&	0.155 	&	0.160 	&	0.293 	&	0.036 \\
10000	&	0.170 	&	0.163 	&	0.171 	&	0.353 	&	0.034 	&	&	0.171 	&	0.164 	&	0.167 	&	0.342 	&	0.036 \\
15000	&	0.170 	&	0.166 	&	0.170 	&	0.392 	&	0.034 	&	&	0.163 	&	0.152 	&	0.158 	&	0.379 	&	0.036 \\
20000	&	0.169 	&	0.164 	&	0.172 	&	0.453 	&	0.034 	&	&	0.159 	&	0.148 	&	0.160 	&	0.409 	&	0.036 \\
\hline

&\multicolumn{5}{c}{{\tt HE}}&&\multicolumn{5}{c}{{\tt HE}}\\\cline{2-6}\cline{8-12}
5000	&	0.156 	&	0.158 	&	0.155 	&	0.305 	&	0.040 	&	&	0.153 	&	0.148 	&	0.153 	&	0.276 	&	0.041 \\
10000	&	0.158 	&	0.158 	&	0.163 	&	0.339 	&	0.041 	&	&	0.158 	&	0.157 	&	0.160 	&	0.327 	&	0.041 \\
15000	&	0.158 	&	0.161 	&	0.162 	&	0.381 	&	0.042 	&	&	0.149 	&	0.147 	&	0.152 	&	0.367 	&	0.041 \\
20000	&	0.157 	&	0.160 	&	0.163 	&	0.445 	&	0.041 	&	&	0.145 	&	0.141 	&	0.149 	&	0.401 	&	0.042 \\
\hline

&\multicolumn{5}{c}{{\tt Time}(in seconds)}&&\multicolumn{5}{c}{{\tt Time}(in seconds)}\\\cline{2-6}\cline{8-12}
5000	&	0.676 	&	1.220 	&	0.576 	&	0.181 	&	0.107 	&	&	0.688 	&	1.247 	&	0.528 	&	0.182 	&	0.122 \\
10000	&	6.759 	&	5.903 	&	3.222 	&	0.806 	&	0.456 	&	&	6.335 	&	6.202 	&	2.811 	&	0.778 	&	0.489 \\
15000	&	15.37 	&	14.64 	&	7.296 	&	1.843 	&	1.080 	&	&	14.85 	&	13.03 	&	6.982 	&	1.868 	&	1.159 \\
20000	&	29.14 	&	26.08 	&	14.09 	&	3.597 	&	2.344 	&	&	28.80 	&	23.29 	&	14.06 	&	3.412 	&	2.236 \\
			\hline
		\end{tabular}
	\end{center}  
\end{table}
\section{Conclusion}
The Heaviside step function ideally characterizes the binary status of some real-world data. However, the hardness stemmed from the dis-continuity restricts its applications for a time. Fortunately, this paper manages to address the optimization \eqref{ell-0-1} with the Heaviside set constraint. One of the key factors of such a success is based on the closed form of the normal cone to the feasible set, namely, the Heaviside step set \eqref{l01-s}. Another key factor is the establishment of the $\tau$-stationary point, which enables us to benefit from the Newton type method. We feel that those results could be extended to a more general case  where $A\bfx-\bfb$ in the problem \eqref{ell-0-1} is replaced by some nonlinear functions $c(\bfx)$. A possible explanation can be given as follows: For the optimization problem $\min \{f(\bfx):c(\bfx)\leq0\}$,  if only a few (e.g., $s\ll m$) inequalities allow to be violated in the constraints, then $\|(c(\bfx))_+\|_0\leq s$.  Moreover, it is also worth applying the Heaviside set constrained optimization into dealing with some other  relevant problems, such as  the maximum rank correlation   estimation for the linear transformation regression in statistics\cite{han1987non,sherman1993limiting, lin2013smoothed} and the area under the receiver operating characteristic curve in medicine  \cite{ma2005regularized, pepe2006combining, ma2007combining, zhao2011auc,ghanbari2019novel}.

\section*{Acknowledgements}
 This work was funded by the the National Science Foundation of China (11971052, 11801325) and Young Innovation Teams of Shandong Province (2019KJI013).

%
%

\bibliographystyle{abbrv}
\bibliography{references}

\section{Appendix}
 
Similar to the definitions in \eqref{constants-y-H}, we denote
 \beq\label{zab}
\bfw :=(\bfx;\bfl),~~\bfy:=A\bfx-\bfb,~~~ \bfz: =  \bfy+\tau \bfl.
\eeq
To prove  \cref{the:quadratic}, we first prove the following lemma. 
  \begin{lemma}\label{lemma-neighbour} Let $\bfw^*$ be a \ts\ with $0<\tau<\tau_*$ of the problem (\ref{ell-0-1}) and  $c_*, C_*, \tau_*$ be  given by (\ref{constants}). The following results hold.  
\begin{itemize}
\item[a)] There is a $\delta_*>0$ such that for any $\bfw\in U(\bfw^*,\delta_*)$ and any $T\in\T_\tau(\bfz;s)$, 
 \beq\label{gw*-0} F(\bfw^*;T)=0. \eeq
 \item[b)] Suppose Assumptions \ref{ass-1} and \ref{ass-2} hold. There   always exists a $\delta_*>0$ such that for any $\bfw\in U(\bfw^*,\min\{ {\delta_*},{1}/({c_* L_*})\})$ and any $T\in\T_\tau(\bfz;s)$, 
 \beq\label{bd-h-0}
0<1/c_* \leq \sigma_{\min}(\nabla F(\bfw;T))~\leq \|\nabla F(\bfw;T)\| ~\leq   C_*. 
\eeq
\end{itemize} 
\end{lemma}
\begin{proof} a) Recall $J_*$ in \eqref{gamma*} and  $\bfy^*=A\bfx^*-\bfb$ by \eqref{constants-y-H} that $$J_*=\{i\in\N_m:(A\bfx^*-\bfb)_i = 0\}=\{i\in\N_m:~y^*_i = 0\}.$$  This and $\bfw^*$ being a \ts\ with  $0<\tau<\tau_*$  satisfying \eqref{eta-point-1} lead to
\begin{eqnarray}\label{lambda-i-j}
\lambda_i^*=0,~y_i^*\neq0,~i\in\overline J_*,~~~~
  0\leq \lambda_i^*\leq y^*_{[s]}/\tau,~y_i^*=0,~i\in J_*. 
\end{eqnarray}
Based on which, decompose $\bfz^*$ as 
\begin{eqnarray}\label{z-decom}
 \bfz^*=\bfy^*+\tau\bfl^*=\left[  
     \begin{array}{ccc}
    \bfy^*_{J_*}+\tau\bfl^*_{J_*}\\
   \bfy^*_{\overline J_*}+ \tau\bfl^*_{\overline{J}_*}
     \end{array}
     \right]=\left[  
     \begin{array}{r}
    \tau\bfl^*_{J_*}\\
    \bfy^*_{\overline{J}_*}
     \end{array}
     \right].
\end{eqnarray}
Note that $y^*_{[s]}>0$ and $\max_{i} \lambda_i^* >0$ if $\bfl\neq0$ by \eqref{lambda-i-j}, so $\tau_*$ in  \eqref{constants} is well defined, which together with
 $0<\tau<\tau_*$ indicates 
\begin{eqnarray}\label{y*sl}{y^*_{[s]}}-\tau\max_{i} \lambda_i^*>{y^*_{[s]}}-\tau_*  \max_{i} \lambda_i^* \overset{\eqref{constants}}{=} 0.\end{eqnarray}
For any $T_*\in\T_\tau(\bfz^*;s)$,  we have  $\N_m=\Gamma_+^*\cup\Gamma_-^*\cup\Gamma_0^*$ by \eqref{notation-z} and  
\begin{eqnarray}\label{T-G-G-G}
 T_*=(\Gamma_+^*\setminus\Gamma_s^*)\cup\Gamma_0^*,~~ \overline T_*=\Gamma_s^*\cup\Gamma_-^*,~~\Gamma_s^* \subseteq \Gamma_+^*.
\end{eqnarray}
From \eqref{z-decom} and \eqref{y*sl},  $\overline{J}_*$ contains all  indices of $s$ largest elements  and all indices of  negative elements of $\bfz^*$ if  $\|\bfy_+^*\|_0=s$, namely, 
\begin{eqnarray}\label{G*y*}\overline{J}_*= \overline T_*,~~|\Gamma^*_s|=\|\bfy^*_+\|_0,~~J_*=T_*.\end{eqnarray}
It is easy to check that \eqref{G*y*} also holds if  $\|\bfy_+^*\|_0<s$ (Under such a case $\Gamma^*_s=\Gamma^*_+$ and $T_*= \Gamma_0^*$). 
Similarly, for any $T\in\T_\tau(\bfz;s)$, we have $\N_m=\Gamma_+\cup\Gamma_-\cup\Gamma_0$ and
\begin{eqnarray}
\label{TGG-G-J} T = (\Gamma_+\setminus\Gamma_s) \cup \Gamma_0,~~ \overline{T}=\Gamma_s \cup \Gamma_-, ~~\Gamma_s \subseteq \Gamma_+.
 \end{eqnarray}  
One can easily show that there is a sufficiently small $\delta_*>0$ (relied on $\bfw^*$) such that for $\bfw\in U(\bfw^*,\delta_*)$ and any $T\in\T_\tau(\bfz;s)$,
\beq\label{JGG}
\Gamma_-^*\subseteq \Gamma_-,~\Gamma_+^*\subseteq \Gamma_+,~\Gamma_s^* \subseteq \Gamma_s
\eeq
due to $z_i^*<0, \forall i\in \Gamma_-^*$ and $z_i^*>0, \forall i\in \Gamma_+^*$.  We just show one of them. If $\Gamma_-^*\subseteq \Gamma_-$ is not true, then there is a $i\in \Gamma_-^*$ but $i\notin \Gamma_-$. This means $z_i\geq0$ and $z_i^*<0$, which deliver $\|\bfz-\bfz^*\|\geq|z_i-z_i^*|\geq|z_i^*|$. On the other hand,  by \eqref{constants-y-H} and \eqref{zab} one can derive $\|\bfz-\bfz^*\| \leq \max\{\|A\|,\tau\} \|\bfw-\bfw^*\|$,  causing a contradiction due to $\delta_*$ being small enough. Therefore, those relations in \eqref{JGG} are true and  imply
\beq\label{JGG1}
\overline{T}_*=(\Gamma_s^*\cup \Gamma_-^*) \subseteq (\Gamma_s\cup \Gamma_-)=\overline{T}, ~~T \subseteq T_*.
\eeq
Moreover, let $\Delta:=(\overline T\setminus \overline T_*)= (  T_*\setminus  T)$. Recall that 
 $T_*=(\Gamma_+^*\setminus\Gamma_s^*)\cup\Gamma_0^*$ from \eqref{T-G-G-G}  and  $T = (\Gamma_+\setminus\Gamma_s) \cup \Gamma_0$ from \eqref{TGG-G-J}. If $|\Gamma_s^*|<s$, then by \eqref{T-z}, we have $\Gamma_s^*=\Gamma_+^*$. So $T_*=\Gamma_0^*$ and $\Delta\subseteq \Gamma_0^*$. If $|\Gamma_s^*|=s$, then $$s=|\Gamma_s^*|\overset{\eqref{JGG}}{\leq}|\Gamma_s|=\min\{s,|\Gamma_+|\}\overset{\eqref{T-z}}{\leq} s.$$ 
 So $\Gamma_s^* = \Gamma_s$ and $(\Gamma_+^*\setminus\Gamma_s^*)\subseteq(\Gamma_+\setminus\Gamma_s) $ by \eqref{JGG}, which also shows  
\beq\label{JGG2} \Delta=(  T_*\setminus  T)\subseteq \Gamma_0^*.\eeq
We are ready to prove that for any $\bfw\in U(\bfw^*,\delta_*)$ and any $T\in \T_{\tau}(\bfz;s)$,
\begin{eqnarray}\label{F-w-T-0}F(\bfw^*;T)= \left[  
     \begin{array}{c}
    \nabla f(\bfx^*)+ A^\top_{  T}\bfl^*_{  T}\\
     A_{ T}\bfx^*  -\bfb_{ T}\\ 
     \bfl^*_{\overline T}\\
     \end{array}
     \right]=0. \end{eqnarray}  
In fact,  it follows from  \cref{sta-eq} and $J_*=T_*$ by \eqref{G*y*} that the \ts\ stationary point $\bfw^*$ of (\ref{ell-0-1}) satisfies the conditions
 \allowdisplaybreaks    \begin{eqnarray}
 \nabla f(\bfx^*)+ A^\top_{ T_*}\bfl^*_{  T_*}&=&0,\nonumber\\ 
\label{sta-eq-1-1}     A_{  T_*}\bfx^*-\bfb_{  T_*}&=&0,   \\ 
     \bfl^*_{\overline T_*}=0,~~
     \bfz^*_{\Gamma_0^*}&=&0,\nonumber
  \end{eqnarray}
for any $T_*\in \T_{\tau}(\bfz^*;s)$, where the last one is from the definition of $\Gamma_0^*$ in \eqref{notation-z}. These conditions enable us to obtain three facts: 
\begin{itemize}
\item $A_{  T}\bfx^*-\bfb_{  T}=0$  due to $T\subseteq T_*$ by \eqref{JGG1}. 
\item  As $\Delta\subseteq \Gamma_0^* \subseteq T_*$ by \eqref{JGG2}, we have $ \bfz^*_{\Delta}=0, (A\bfx^*-\bfb)_{ \Delta}=0$ and thus $\tau\bfl^*_{\Delta}= \bfz^*_{\Delta}- (A\bfx^*-\bfb)_{ \Delta}=0$. This suffices to $\bfl^*_{\overline T}=\bfl^*_{\overline T_* \cup (\overline T\setminus\overline T_*)}=\bfl^*_{\overline T_* \cup \Delta}=0$.
\item By $\bfl^*_{\overline T}=0$ and $\bfl^*_{\overline T_*}=0$,  
\begin{eqnarray*}
\nabla f(\bfx^*)+ A_{  T}^\top\bfl^*_{ T}=\nabla f(\bfx^*)+ A ^\top\bfl^*  = \nabla f(\bfx^*)+ A_{ T_*}^\top\bfl^*_{  T_*}=0.
\end{eqnarray*} 
\end{itemize}
Overall, the above three facts verify \eqref{F-w-T-0}, claiming the conclusion.  \\

b)  For any two matrices $D$ and $D'$, we have
  \begin{eqnarray}\label{D-D}
 \|D'-D\|
&\geq& \max_{i}|\sigma_i( D')-\sigma_i(D)|\nonumber\\
&\geq& |\sigma_{i_0}(D')-\sigma_{i_0}(D)|\nonumber\\
&\geq& \sigma_{i_0}(D')-\sigma_{\min}(D)\nonumber\\
&\geq& \sigma_{\min}(D')-\sigma_{\min}(D), 
\end{eqnarray}
where the first inequality holds from \cite[Reminder (2), on Page 76]{lutkepohl1996handbook} and  
$i_0$ satisfies $\sigma_{i_0}(D)=\sigma_{\min}(D)$. Recall $H(J)$ in \eqref{constants-y-H} that
 \beq\label{FH*-FG}  H(J) \overset{\eqref{constants-y-H} }{=}\left[  
     \begin{array}{ccc}
    \nabla^2 f(\bfx^*)& A^\top_{J}\\
     A_{J}&0
     \end{array}
     \right], ~~~~ \nabla F(\bfw^*;T) =\left[  
     \begin{array}{cc}
    H( T)&0\\
     0&I  
     \end{array}
     \right].~~~~~\eeq    
It follows from the full row rankness of $A_{J_*}$ in Assumption  \ref{ass-1} that  $A_{J}$ is full row rank for any $J\subseteq J_*$. Again by Assumption  \ref{ass-1}, we have $H(J)$ is non-singular for any $J\subseteq J_*$, namely $\sigma_{\min} (H(J))>0$. As a result,
 \beq\label{1-c-*}1/c_*\overset{\eqref{constants}}{=}\min\{ \min_{J\subseteq J_*} \sigma_{\min} (H(J)),1\}/2>0. \eeq
From a), there is  a $\delta_*>0$ such that for any $\bfw\in U(\bfw^*,\delta_*)$ and any $T\in\T_\tau(\bfz;s)$
 \beq\label{JJoG0-0}
T\overset{\eqref{JGG1}}{\subseteq}T_*\overset{\eqref{G*y*}}{=}J_*,
 \eeq
which means $H( T)$ is a submatrix of $H(J_*)$. Then by \cite[Theorem 1]{thompson1972principal} that the maximum  singular value of a  matrix is no less than the maximum singular value of its sub-matrix, we obtain
  \beq\label{JJoG0-H}
\|H( T)\| \leq \|H(J_*)\|,~~ \sigma_{\min}(H( T)) \overset{\eqref{JJoG0-0}}{\geq} \min_{J\subseteq J_*} \sigma_{\min} (H(J)).
  \eeq
For any $\bfw\in U(\bfw^*,\min\{ {\delta_*},{1}/({c_* L_*})\})$, the locally Lipschitz continuity of $\nabla^2 f$  around $\bfx^*$ with $L_*>0$ leads to
 \beq\label{HH*2}
 \|\nabla F(\bfw;T)- \nabla F(\bfw^*;T)\|
 &\overset{\eqref{sta-eq-2}}{=}& \|\nabla^2 f(\bfx)- \nabla^2 f(\bfx^*)\|\nonumber\\
 & \overset{\eqref{Hessian-Lip}}{\leq}& L_* \|   \bfx-\bfx ^*\|\leq L_* \|   \bfw-\bfw^*\| < 1/c_*, ~~\eeq
which contributes to
 \begin{eqnarray*}
\|\nabla F(\bfw;T)\|&\leq& \|\nabla F(\bfw^*;T) \| + \|\nabla F(\bfw;T)- \nabla F(\bfw^*;T)\| \nonumber\\
&\overset{\eqref{HH*2}}{\leq}&\|\nabla F(\bfw^*;T) \| + 1/c_*\nonumber\\
& \overset{\eqref{FH*-FG}}{=}& \max\{ \|H( T)\|,1\} + 1/c_*\nonumber\\
  &\overset{\eqref{JJoG0-H}}{\leq}& \max\{ \|H(J_*)\|,1\} + 1/c_* \nonumber\\  
&\leq&   2\max\{ \|H(J_*) \|,1\}     \overset{\eqref{constants}}{=} C_*. 
 \end{eqnarray*}
Next we show  the smallest singular value of $\nabla F(\bfw;T)$ has a lower bound. In fact, 
 \beq \sigma_{\min}(\nabla F(\bfw;T))
  & \overset{\eqref{D-D}}{\geq}&\sigma_{\min}(\nabla F(\bfw^*;T))  -\|\nabla F(\bfw^*;T)-\nabla F(\bfw;T)\|  \nonumber\\
  &\overset{\eqref{HH*2}}{\geq}&\sigma_{\min}(\nabla F(\bfw^*;T))-1/c_*\nonumber\\
 &\overset{\eqref{FH*-FG}}{=}& \min\{ \sigma_{\min}(H(  T)),1 \}  -1/c_* \nonumber\\
 & \overset{\eqref{JJoG0-H}}{\geq}&\min\{\min_{J\subseteq J_*} \sigma_{\min} (H(J)),1\}   -1/c_* \nonumber\\
 &\overset{\eqref{1-c-*}}{=}&  2/c_*-1/c_* =  1/c_*. \nonumber \eeq 
Hence, the whole proof is completed. 
\qed 
\end{proof}

\subsection*{The proof of  \cref{the:quadratic}}

\begin{proof} a) Let $\epsilon_*:= \min\{ {\delta_*},{1}/({c_* L_*})\}$. Then
\beq\label{radius} \epsilon_*c_* L_* \leq 1.
\eeq
It follows from \cref{lemma-neighbour} and $\bfw^0\in U(\bfw^*,\epsilon_*)$ that
\beq\label{gw*-bd-h}F (\bfw^{*};T_0) \overset{\eqref{gw*-0}}{=}0, ~~~\|(\nabla F (\bfw^{0};T_0))^{-1}\|\overset{\eqref{bd-h-0}}{\leq} c_* 
\eeq
for the $T_0\in\T_\tau(\bfz^0;s)$. From \eqref{newton-dir-0}, we have
\beq\label{dkt}
\nabla F (\bfw^{0};T_0)~\bfd^{0}=- F (\bfw^{0};T_0).
\eeq
 \cref{lemma-neighbour} c) states that $\nabla F (\bfw^{0};T_0)$ is non-singular and thus $\bfd^{0}$ is well defined.
Let   $\bfw^{0}_\beta= \bfw^{*} + \beta(\bfw^{0}-\bfw^{*})$ where $\beta\in[0,1]$. One can easily check that $\bfw^{0}_\beta\in U(\bfw^*,\epsilon_*)$ as $\|\bfw^{0}_\beta-\bfw^{*}\|= \beta\|\bfw^{0}-\bfw^{*}\|\leq  \epsilon_*.$ For notational simplicity, let $\nabla F^0:=\nabla F (\bfw^{0};T_0)$ and $\nabla F^0_\beta:=\nabla F (\bfw^{0}_\beta;T_0)$.The locally Lipschitz continuity of $\nabla^2 f(\cdot)$  around $\bfx^*$ with $L_*>0$ brings out
\beq\label{lipschitz}
\|\nabla F^0-\nabla F^0_\beta \|&\overset{\eqref{sta-eq-2}}{ =}&
\|\nabla ^2 f (\bfx^{0})- \nabla ^2 f (\bfx^{0}_\beta) \|\nonumber\\
& \overset{\eqref{Hessian-Lip}}{\leq}& L_* \|\bfx^{0} -\bfx^{0}_\beta\| \leq L_*  \|\bfw^{0} -\bfw^{0}_\beta\| \nonumber\\
&=&L_* (1-\beta)\|\bfw^{0} -\bfw^{*}\|.
\eeq
 Note that for the fixed $T_0$, the function $F(\cdot;T_0)$ is differentiable. So we have the following mean value expression
\beq
F (\bfw^{0};T_0)&=&   F (\bfw^{*};T_0)+\int_0^1  \nabla F^0_{\beta}\cdot(\bfw^{0}-\bfw^{*})d\beta\nonumber\\
\label{mean-value}&\overset{\eqref{gw*-bd-h}}{=}&\int_0^1  \nabla F^0_{\beta}\cdot(\bfw^{0}-\bfw^{*})d\beta. \eeq
Then the following chain of inequalities hold.
{\allowdisplaybreaks \beq\label{quadratic-chain}
\|\bfw^{1}-\bfw^{*}\| &= &\|\bfw^{0}+\bfd^{0}-\bfw^{*}\|\nonumber\\
  &\overset{\eqref{dkt}}{=} &\|\bfw^{0}-\bfw^{*}-(\nabla F^0)^{-1}  \cdot F (\bfw^{0};T_0) \|\nonumber\\
&\overset{\eqref{gw*-bd-h}}{\leq} &c_*\|\nabla F^0\cdot(\bfw^{0}-\bfw^{*})-  F (\bfw^{0};T_0) \|\nonumber\\
&\overset{\eqref{mean-value}}{=}&c_*\|\nabla F^0\cdot(\bfw^{0}-\bfw^{*})-  \int_0^1 \nabla  F^0_\beta\cdot(\bfw^{0}-\bfw^{*})d\beta \|\nonumber\\
& \leq &c_* \int_0^1  \|\nabla F^0-  \nabla F^0_\beta \|\|\bfw^{0}-\bfw^{*} \|d\beta \nonumber\\
&\overset{\eqref{lipschitz}}{\leq}&c_*L_* \int_0^1  (1-\beta) \|\bfw^{0} -\bfw^{*}\|\|\bfw^{0}-\bfw^{*} \|d\beta \nonumber\\
& = &0.5 c_* L_* \|\bfw^{0}-\bfw^{*} \|^2.
\eeq} 
The above relation suffices to
\beq 
\|\bfw^{1}-\bfw^{*}\|  \leq 
 0.5 c_* L_* \epsilon_* \|\bfw^{0}-\bfw^{*} \| \overset{\eqref{radius}}{\leq} 0.5\|\bfw^{0}-\bfw^{*} \| < \epsilon_*.
 \eeq
This means $\bfw^1\in U(\bfw^*,\epsilon_*)$. Replacing $T_0 $ by $T_1$, the same reasoning allows us to show that $\bfd^{1}$ is well defined and $\|\bfw^{2}-\bfw^{*}\|  \leq 0.5 c_* L_*  \|\bfw^{1}-\bfw^{*} \|^2 $. By the induction, we can conclude that $\bfw^k\in U(\bfw^*,\epsilon_*)$, $\bfd^{k}$ is well defined and 
\beq\label{quadratic-rate}
\|\bfw^{k+1}-\bfw^{*}\|  &\leq& 0.5 c_* L_*  \|\bfw^{k}-\bfw^{*} \|^2, \\
\label{quadratic-rate-1}&\leq& 0.5 c_* L_* \epsilon_*  \|\bfw^{k}-\bfw^{*} \| 
\overset{\eqref{radius}}{\leq} 0.5 \|\bfw^{k}-\bfw^{*} \|. 
\eeq
Therefore, \eqref{quadratic-rate} claims b). The conclusion of a) can be made by \eqref{quadratic-rate-1}  that $\bfw^{k}\rightarrow\bfw^{*}$ and $\bfd^k=\bfw^{k+1}-\bfw^{k}=\bfw^{k+1}-\bfw^{*}+\bfw^{*}-\bfw^{k}\rightarrow0.$
 \qed
\end{proof} 

\end{document}